\documentclass[oneside]{amsart}

\usepackage{amsmath}
\usepackage{amsthm}
\usepackage{mathtools}
\usepackage{amssymb}
\usepackage{latexsym}
\usepackage{graphics}
\usepackage{graphicx}
\usepackage{epstopdf}
\usepackage{tikz}
\usepackage{color}
\usepackage{import}
\usepackage{epstopdf}
\usepackage{float}
\usepackage[normalem]{ulem}

\usepackage{caption}
\usepackage[labelformat=simple]{subcaption}

\usepackage[hidelinks,pdftex]{hyperref}

\AtBeginDocument{%
   \def\MR#1{}
}

\usepackage{marginnote}
\long\def\@savemarbox#1#2{\global\setbox#1\vtop{\hsize\marginparwidth
  \@parboxrestore\tiny\raggedright #2}}
\numberwithin{equation}{section}
\theoremstyle{plain}
\newtheorem{theorem}[equation]{Theorem}
\newtheorem{corollary}[equation]{Corollary}
\newtheorem{lemma}[equation]{Lemma}

\newtheorem{question}[equation]{Question}

\newtheorem*{namedtheorem}{\theoremname}
\newcommand{\theoremname}{testing}
\newenvironment{named}[1]{\renewcommand{\theoremname}{#1}\begin{namedtheorem}}{\end{namedtheorem}}

\theoremstyle{definition}
\newtheorem{definition}[equation]{Definition}
\newtheorem{remark}[equation]{Remark}

\newtheorem{assumption}[equation]{Assumption}

\usepackage{geometry}
\newcommand{\HH}{{\mathbb{H}}}

\newcommand{\QQ}{{\mathbb{Q}}}

\newcommand{\from}{\colon} % As in ``f maps _from_ X _to_ Y''.

\newcommand{\area}{\operatorname{area}}
\newcommand{\bdy}{\partial}
\renewcommand{\setminus}{{\smallsetminus}}
\newcommand{\len}{\operatorname{len}}

\newcommand{\refthm}[1]{Theorem~\ref{Thm:#1}}
\newcommand{\reflem}[1]{Lemma~\ref{Lem:#1}}

\newcommand{\refcor}[1]{Corollary~\ref{Cor:#1}}

\newcommand{\refeqn}[1]{\eqref{Eqn:#1}}

\newcommand{\refsec}[1]{Section~\ref{Sec:#1}}
\newcommand{\reffig}[1]{Figure~\ref{Fig:#1}}

\title[Polynomial bounds for surfaces]{Polynomial bounds for surfaces in cusped 3-manifolds}
\author{Jessica S. Purcell}
\author{Anastasiia Tsvietkova}
\subjclass[2020]{}
\everymath{\displaystyle}

\begin{document}

\begin{abstract}
It is natural to ask how many isotopy classes of embedded essential surfaces lie in a given 3-manifold. The first bounds on the number of such surfaces were exponential, using normal surfaces. More recently, by restricting to alternating link complements in 3-sphere, Hass, Thompson and Tsvietkova obtained polynomial bounds, but for a limited class of surfaces: closed and spanning ones. Here, we complete the picture for classical alternating links and extend these results to other classes of cusped 3-manifolds. We give explicit polynomial bounds on all embedded essential surfaces, closed or any boundary slope, orientable or non-orientable. Our 3-manifolds are complements of links with alternating diagrams on wide classes of surfaces in broad families of 3-manifolds. This includes all alternating links in 3-sphere as well as many non-alternating ones, alternating virtual knots, many toroidally alternating knots, and most Dehn fillings of such manifolds. 
\end{abstract}

\maketitle

%%%%%%%%%%%%%%%%%%%%%%%%%%%%%%%%%%%%%%%%%%%%%%%%%%%%%%%%%%%%%%%%%
\section{Introduction}\label{Sec:Intro}

For decades, the study of essential surfaces in 3-manifolds has led to important topological and geometric consequences. For example, the existence or nonexistence of low genus essential surfaces gives crucial insight into the geometrization of the 3-manifold, due to work of Thurston~\cite{Thurston:Bulletin}, and every 3-manifold is known to have a finite cover containing an embedded essential surface, by work of Kahn and Markovic~\cite{KahnMarkovic:SfceSgp}, Haglund and Wise~\cite{HaglundWise}, and Agol~\cite{Agol:VirtualHaken}. It is natural to ask how many connected embedded essential surfaces lie in a given 3-manifold for fixed Euler characteristic (up to ambient isotopy). For certain 3-manifolds, the number is infinite; for example there are infinitely many nonisotopic Seifert surfaces in a connect sum of knots, as shown in work of Eisner~\cite{Eisner}. However, for 3-manifolds with no essential embedded spheres or tori, the number of embedded surfaces with fixed topological type is finite; see for example Jaco and Oertel~\cite{JacoOertel}. 

Using triangulations of 3-manifolds and normal surface theory, one can bound the number of isotopy classes of fundamental normal surfaces; see Matveev~\cite{Matveev}, Jaco and Oertel~\cite{JacoOertel}, and Hass, Lagarias, and Pippenger~\cite{HLP}. This can be used to give bounds on the number of all essential surfaces in certain settings, but the bounds are a tower of exponentials in terms of genus. Hass, Thompson, and Tsvietkova recently significantly improved upon these results by restricting to alternating link complements. In this setting, they showed that the number of isotopy classes of closed surfaces is polynomial in the number of crossings of an alternating diagram~\cite{HTT1}. They then generalized to give similar results for spanning surfaces, i.e.\ those surfaces with a single boundary component forming a longitude for the knot~\cite{HTT2}. However, even in the alternating setting, the most general case remained open, namely counting surfaces of any topological type. 

In this paper, we complete the program for alternating links. We find polynomial bounds on the number of isotopy classes of embeddings of any surface, orientable or nonorientable, with or without boundary, and boundary forming any slope on a neighbourhood of the link.
We do this by broadening our perspective. Rather than restricting to tools available in the classical alternating setting, we develop and extend more general tools to cusped 3-manifolds that can be described by similar combinatorial patterns as alternating links. The combinatorics required is an alternating diagram on an arbitrary closed surface lying in an arbitrary 3-manifold, with mild hypotheses, namely weakly generalized alternating links defined below. 
As a consequence, our results immediately extend to give polynomial bounds for broader families of 3-manifolds that have received much recent attention: virtual alternating links, many toroidally alternating links, and most Dehn fillings of these manifolds and classical alternating links.

Our bounds are universal, in the sense that given a fixed topological surface, we obtain the same formula bounding isotopy classes of its embeddings in any of the above 3-manifolds. They are also explicit and effective. The formula is polynomial in the crossing number of a diagram, with the degree of the polynomial and the constants of multiplication depending only on the Euler characteristic of the surface, and most generally, on topology of the ambient 3-manifold.

\subsection{Main results}

Our first main theorem completes the program started by Hass, Thompson, and Tsvietkova~\cite{HTT1,HTT2}. It gives polynomial bounds for surfaces in alternating link complements, with no restrictions on surface.

\begin{named}{\refthm{MainCountSphereMeridCompress}}
Let $\pi(L)$ be a prime alternating projection of a link $L$ onto $S^2$ in $Y=S^3$ with $n$ crossings. Let $Z$ be a connected topological surface with Euler characteristic $\chi(Z)$, with all boundary components (if any) on $N(L)$.
The number of ways, up to isotopy, that $Z$ can be properly embedded in $S^3-N(L)$ as an essential surface is at most
\[ (6n)^{80\chi(Z)^2} 2^{-4\chi(Z)+2} \]
\end{named}

We also restrict to meridianally incompressible surfaces, obtaining a stronger bound of $(6n)^{80\chi(Z)^2}$ in \refthm{MainCountSphere}. Such surfaces include all quasifuchsian surfaces, giving a bound that applies in a more geometrical setting; see~\refcor{QuasifuchsianSphere}. Quasifuchsian surfaces are known to be very prevalent in cusped hyperbolic 3-manifolds by work of Cooper and Futer~\cite{CooperFuter}, and in closed hyperbolic 3-manifolds by work of Kahn and Markovic~\cite{KahnMarkovic:SfceSgp}. 

The methods to prove \refthm{MainCountSphereMeridCompress} are much more general than the classical alternating setting. They can be used immediately to give bounds in other important settings, such as that of virtual alternating knots.

\begin{named}{\refcor{CountVirtualKnots}}
Let $\pi(L)$ be a virtual link $L$, with a weakly prime alternating projection onto a surface $\Pi=F\times\{0\} \subset Y=F\times[-1,1]$ with $n$ crossings. Fix a connected, orientable topological surface $Z$ with Euler characteristic $\chi$. Then up to isotopy, the number of ways to properly embed $(Z, \bdy Z)$ as an essential, meridianally incompressible surface in $(Y-N(L), \bdy N(L))$  is at most:
\[ (2(g+1))^{-4\chi} \cdot (6n)^{-800\chi^3+80\chi^2} 
\]
If in addition $Y-L$ admits a hyperbolic structure, then the above bounds quasifuchsian surfaces.

If we do not require meridianal incompressibility, the number of ways is at most:
\[ (4(g+1))^{-4\chi(Z)+2}(6n)^{-800\chi^3+80\chi^2} \]
\end{named}

Interest in the complements of knots in spaces besides $S^3$, with diagrams on surfaces besides $S^2$, has been growing in recent years. Some recent results generalize classical knot theory, for example the work of Boden and Karimi~\cite{Boden-Karimi}. Other research concerns quantum invariants, such as~\cite{Boninger, BDKY}, cobordisms, e.g.~\cite{Turaev, CarterKamadaSaito:Virtual}, geometric topology, as in~\cite{Hayashi, Kaplan-Kelly, Ichihara:NoEssSfces}, as well as other aspects of these 3-manifolds~\cite{BodenDancsoLinWinkinson, CarterSilverWilliams}. One widely studied class of such 3-manifolds is virtual links. They were described by Kauffman in the 1990s as natural generalisations of Gauss codes for classical knot diagrams~\cite{Kauffman:VirtualKnots}. They were subsequently shown to be equivalent to diagrams of links in thickened surfaces up to moves called stabilisation and destabilisation~\cite{KamadaKamada:Virtual, CarterKamadaSaito:Virtual}, and Kuperberg showed that there is a unique minimal genus surface $\Pi$ for which the link embeds in $\Pi\times[-1,1]$ with diagram on $\Pi\times\{0\}$, admitting no destabilisations~\cite{Kuperberg:Virtual}. In the alternating case, provided the diagram is weakly prime (definition recalled below), such a link satisfies all conditions to ensure it is an instance of a \emph{weakly generalized alternating link}, defined by Howie and Purcell~\cite{HowiePurcell}.

Indeed, the tools we develop here apply most broadly in the setting of weakly generalized alternating links. Informally, such links are complements of links in general 3-manifolds that have an alternating diagram on some embedded surface, with mild restrictions involving requirements on connectivity and interaction with compression discs for the projection surface.
Our most general result is:

\begin{named}{\refthm{NotMeridIncompr}}
Let $L$ be a weakly generalized alternating link in a 3-manifold $Y$, with alternating diagram $\pi(L)$ on a projection surface $\Pi$, satisfying Assumptions \ref{ManifoldProj}, \ref{CellularRep}.  Suppose that $\pi(L)$ has $n$ crossings, and suppose that for each 3-manifold component $\Sigma$ of $Y-N(\Pi)$, there is a universal bound $X$ on the number of isotopy classes of incompressible surfaces properly embedded in  $\Sigma$ with fixed genus and fixed boundary curves on $\bdy N(\Pi)\cap \bdy\Sigma$. Fix a topological surface $Z$, either orientable or non-orientable, possibly with boundary, with fixed orientable or non-orientable genus $g$ and Euler characteristic $\chi$. Then up to isotopy, the number of ways such a surface can be properly embedded in $Y-N(L)$ as an essential surface is at most:
\[  (2X(g+1))^{-4\chi+2}\cdot (6n)^{-800\chi^3+80\chi^2} \]
\end{named}

The factor $X$ is explicitly known for a number of families of cusped 3-manifolds. For example, it can be used to give bounds for weakly generalized alternating links with a projection surface $\Pi$ that is a Heegaard torus in a lens space, thickened torus, or solid torus; the bound is \refcor{TorusChunks}. These classes of 3-manifolds appear frequently in low-dimensional topology. For example, they include instances of toroidally alternating links introduced by Adams~\cite{Adams:ToroidAlt}.

Again, a direct corollary is a bound on the number of quasifuchsian surfaces in the hyperbolic case; see \refcor{QuasifuchsianGeneral}. 

Finally, because our bounds apply to all topological surfaces, with any boundary slope on the knot or link complement, we are also able to obtain bounds on the number of essential surfaces in 3-manifolds obtained by Dehn filling such link complements. See \refthm{DehnFilling} and \refcor{DehnFilling}.

%%%%%%%%%%%%%%%%%%%%%%%%%%%%%%%%%%%%%%%%%%%%%%%%%%%%%%%%%%%%%%%%%
\subsection{Comparison to other results}
Analogues of the question of how many connected essential surfaces embed in a given 3-manifold have been addressed by others. Masters~\cite{Masters} and then Kahn and Markovic~\cite{KahnMarkovic} found exponential bounds on the number of immersed essential surfaces in 3-manifolds that are hyperbolic and closed. More recently, Dunfield, Garoufalidis, and Rubinstein showed that quasi-polynomial upper bounds exist for connected, orientable surfaces for a certain wide class of cusped 3-manifolds that includes hyperbolic knot complements in the 3-sphere~\cite{DunfieldGaroufalidisRubinstein}. More precisely, they give an algorithm producing a quasi-polynomial count for disconnected surfaces in a fixed 3-manifold. However, all these bounds depend on the 3-manifold. When the 3-manifold changes, the expression of the bound changes in a way that is not well-understood, while the order of dependence on the genus or Euler characteristic stays the same. In contrast, our bounds are explicit and universal, depending on crossing number, and, most generally, the bound on surfaces $X$ from the ambient 3-manifold. We achieve this through developing techniques that are completely different from those of ~\cite{KahnMarkovic} and ~\cite{DunfieldGaroufalidisRubinstein}.

The methods we develop here are also new: they have not been used in any of the previous bounds. Some of our techniques can be seen broadly as generalizations of normal surface theory, though we do not triangulate the 3-manifold, looking instead at a more general decomposition. Moreover, the direct application of classical normal surface theory to this problem does not yield polynomial bounds, as explained above.

One dimension lower, a similar question was addressed by Mirzakhani for curves on surfaces~\cite{Mirzakhani}, with numerous extensions, for example~\cite{Sapir:NonsimpleGeodesics, ErlandssonSouto:MultiGeodesics}. 

A related, but simpler problem is bounding the number of incompressible surfaces that can be simultaneously disjointly embedded in a 3-manifold. It was originally studied by Kneser, who obtained bounds in 1929~\cite{Kneser}. Here, we bound all essential embeddings, not just disjoint collections.

\subsection{How good are these bounds?}

The upper bounds of this paper are likely very far from sharp: our goal is to establish polynomial bounds rather than obtain the most sharp bounds, which would require very delicate analysis at certain stages of the proof. Rather than pursue this analysis, we take frequent shortcuts that preserve the polynomial order of the bound, but at the cost of growing constants and exponents. However, we note that this is still state of the art: there are no known sharp bounds for connected surfaces in any setting, and we have not come across even open conjectures on sharpness in the literature. What we do provide are explicit, effective, computable, and polynomial bounds in the classical, virtual, and toroidally alternating settings. This is novel. Many of the other important results above only prove the existence of certain constants without deriving formulae, where our bounds are immediately computable.

\subsection{Organisation} In \refsec{WGA}, we recall the definition of a weakly generalized alternating link, and the decomposition of its complement into \emph{chunks} introduced by Howie and Purcell~\cite{HowiePurcell}. We also review certain techniques for essential surfaces in these link complements: normal form with respect to a chunk decomposition, combinatorial area, and results from Purcell and Tsvietkova~\cite{Paper1} giving restrictions on how such surfaces meet chunks. 
We begin the count of surfaces in \refsec{Counting}, bounding the number of subsurfaces making up an essential surface, and the number of labels on the boundaries of subsurfaces. In \refsec{Intersections}, we restrict the possible number of boundary curves of subsurfaces.

At this point, we are ready for our first application: the count of essential surfaces for classical alternating links in $S^3$ whose boundary follows a longitude more than one time. We do this in  \refsec{Classical}.

We then return to the more general case of links in any compact orientable 3-manifold, giving a general bound depending on the chunks in \refsec{SfcCount}, and giving applications to certain families of weakly generalized alternating links in \refsec{Tori}: bounds for weakly generalized alternating links in thickened tori and in lens spaces, and bounds for virtual alternating knots that have a diagram that is cellular, meaning all regions are disks. 
Finally, we give applications to Dehn fillings in \refsec{DehnFillings}.

\subsection{Acknowledgments}
Purcell was partially supported by the Australian Research Council, grant DP210103136. Tsvietkova was partially supported by the National Science Foundation (NSF) of the United States, grants DMS-1664425 (previously 1406588), DMS-2005496, and DMS-2142487 (CAREER), as well as the Institute of Advanced Study under NSF grant DMS-1926686 and by the Sydney Mathematical Research Institute. Both authors were partially supported by the Okinawa Institute of Science and Technology, where this research was started.

%%%%%%%%%%%%%%%%%%%%%%%%%%%%%%%%%%%%%%%%%%%%%%%%%%%%%%%%%%%%%%%%%
\section{Preliminaries}\label{Sec:WGA}

In recent work of Hass, Thompson, and Tsvietkova~\cite{HTT1, HTT2}, surfaces are put into a \emph{standard form}, extending work of Menasco~\cite{Menasco1984,Menasco1985} and Menasco--Thistlethwaite~\cite{MenascoThistlethwaite}. 
Here, we combine standard form techniques with related work of Lackenby~\cite{lac00}, Futer--Gu\'eritaud~\cite{fg09}, and Howie--Purcell~\cite{HowiePurcell}, who put surfaces into \emph{normal form} with respect to a decomposition of the knot complement. The combination of these two ideas extends beyond usual alternating links projected onto $S^2$ lying in the 3-sphere. Howie and Purcell extend such tools to apply to a class of knots with alternating diagram projected onto any closed orientable surface $\Pi$ embedded in any irreducible, boundary irreducible, compact orientable 3-manifold $Y$~\cite{HowiePurcell}. These are called \emph{weakly generalized alternating links}. They include alternating links in 3-sphere, which we refer to as \emph{classical alternating links}. They also include many virtual alternating knots, many alternating knots on Heegaard surfaces in closed irreducible 3-manifolds, and other broad families of 3-manifolds.

In this section, we recall the definition of weakly generalized alternating links, the decomposition of their complement into chunks, and normal surfaces within them. Most of this originally appeared in Howie--Purcell~\cite{HowiePurcell} and in Purcell--Tsvietkova~\cite{Paper1}.

\subsection{Weakly generalized alternating links}

\begin{assumption}\label{ManifoldProj}
Throughout, we work with PL manifolds. Let $Y$ be a compact, orientable, irreducible 3-manifold, possibly with boundary. Embedded in $Y$ is a closed, orientable surface $\Pi$. If $Y$ has boundary, we require $\bdy Y$ to be incompressible in $Y-N(\Pi)$, where $N(\cdot)$ denotes an open regular neighborhood. We also require $Y-\Pi$ to be irreducible.
\end{assumption}

A \emph{generalized diagram} is the projection $\pi\from L \to N(\Pi)$ of a link $L$ onto the surface $\Pi$ in general position. That is, $L$ can be isotoped through $Y$ to lie in $N(\Pi)$. The image of the projection $\pi(L)$ consists of crossings and arcs between them on $\Pi$.
Observe that the requirement that $Y-\Pi$ be irreducible means that if $\Pi$ is the 2-sphere, then $Y=S^3$ and the generalized diagram is actually the standard diagram of a knot in the 3-sphere.

A generalized diagram is \emph{alternating} if, for each region of $\Pi\setminus \pi(L)$, each boundary component of the region is alternating. That is, it can be oriented such that crossings run from under to over in the direction of orientation. An alternating generalized diagram is \emph{checkerboard colored} if each region of $\Pi\setminus \pi(L)$ is oriented so that the induced orientation on the boundary is alternating. See~\cite[Figure~1]{Paper1} for examples and nonexamples.

To ensure that the diagram $\pi(L)$ is sufficiently reduced, we introduce a notion of prime. A generalized diagram is \emph{weakly prime} if whenever a disk $D$ embedded in $\Pi$ has boundary $\bdy D$ meeting $\pi(L)$ transversely exactly twice, either the disk contains no crossings in its interior, or $\Pi$ is the 2-sphere and there is a single embedded arc with no crossings in the complementary disk $\Pi-D$. Note that by this definition, a classical alternating link on $S^2$ in $S^3$ that is prime is also weakly prime.

Finally, the representativity of a generalized diagram is defined as follows. Because $\Pi$ is orientable, $Y-N(\Pi)$ has two boundary components for each component $\Pi_i$ of $\Pi$, call them $\Pi_i^+$ and $\Pi_i^-$. Let $r^+(\pi(L), \Pi_i)$ denote the minimum number of intersections between $\pi(L)$ and the boundary of any essential compresssing disk in $Y-N(\Pi)$ whose boundary lies on $\Pi_i^+$; if there are no such essential compressing disks, we definite this to be $\infty$. Define $r^-(\pi(L),\Pi_i)$ similarly. The \emph{representativity} $r(\pi(L), \Pi)$ is the minimum of all values $r^+(\pi(L),\Pi_i)$ and $r^-(\pi(L),\Pi)$ over all $i$. Thus it measures the minimum number of times the boundary of any essential compressing disk for  $Y-N(\Pi)$ meets $\pi(L)$.

The \emph{hat-representativity} $\hat{r}(\pi(L),\Pi)$ is defined to be the minimum of
\[ \bigcup_i \max\{r^+(\pi(L),\Pi_i), r^-(\pi(L),\Pi_i)\}. \]
Thus it measures the minimal number of intersections of the boundary of any essential compressing disk on one side of $\Pi$. See~\cite[Example~2.2]{Paper1} for examples.

Note that in the case $\Pi$ is the 2-sphere inside $Y=S^3$, there are no essential compressing disks for the balls $S^3-N(\Pi)$. Hence the representativity and hat-representativity in the classical alternating setting are both infinite.

A generalized diagram $\pi(L)$ on $\Pi$ is \emph{weakly generalized alternating}~\cite{HowiePurcell} if
\begin{enumerate}
\item $\pi(L)$ is alternating on $\Pi$,
\item $\pi(L)$ is weakly prime,
\item $\pi(L)$ meets each component of the projection surface $\Pi$,
\item each component of $\pi(L)$ projects to at least one crossing in $\pi(L)$,
\item $\pi(L)$ is checkerboard colorable, and
\item the representativity $r(\pi(L),\Pi)\geq 4$.
\end{enumerate}

From now on, every link we consider will be weakly generalized alternating. A classical reduced, prime, alternating diagram of a link $L$ on $\Pi=S^2$ in $S^3$ is an example of a weakly generalized alternating link. As noted in the introduction, there are many more examples, including alternating links on Heegaard tori in $S^3$ and in lens spaces. When the diagram is \emph{cellular}, meaning all regions of $\Pi-\pi(L)$ are disks, these are examples of toroidally alternating knots~\cite{Adams:ToroidAlt}. Virtual alternating knots are further examples of weakly generalized alternating knots; these are alternating knots in a thickened surface $\Sigma\times (-1,1)$, projected onto $\Pi=\Sigma\times\{0\}$. Their geometry has also received attention recently, for example~\cite{Adams:ThickenedSfces, CKP:Biperiodic}. We will revisit them again in \refsec{Virtual}.

\subsection{Chunk decomposition}

Classical alternating knot and link complements have a well-known decomposition into topological polyhedra. This was suggested by W.~Thurston and described by Menasco~\cite{men83}; see also Lackenby~\cite{lac00} and Purcell~\cite[Chapters~1 and~11]{Purcell:HKT}. A more general decomposition of arborescent knots into angled blocks was defined by Futer and Gu{\'e}ritaud~\cite{fg09}, and this was generalized further by Howie and Purcell to weakly generalized alternating links~\cite{HowiePurcell}. We briefly review the salient points here; see also~\cite[Section~3]{Paper1} for further discussion and examples.

A checkerboard colored diagram has two checkerboard colored surfaces. The white surface comes from regions of $\Pi-\pi(L)$ that are colored white, and connected by twisted bands at each crossing; a similar construction holds for the shaded surface. Note the two surfaces intersect at a crossing in a \emph{crossing arc}, which runs from the overstrand to the understrand. The decomposition of the weakly generalized alternating link complement is obtained by cutting along the two checkerboard surfaces. This cuts $Y-N(L)$ into components with interiors homeomorphic to $Y-N(\Pi)$. These are called \emph{chunks}. Crossing arcs become ideal edges on the chunk boundary; strands of the knot become ideal vertices, and we contract strands so that ideal vertices lie at a crossing of the diagram, and edges follow the diagram graph of $\pi(L)$. Thus each chunk is a connected component of $Y-N(\Pi)$ with $\Pi^+$ and $\Pi^-$ decorated by the following:
\begin{enumerate}
\item Edges of $\pi(L)$, corresponding to ideal edges. We call these \emph{interior edges}. Four interior edges, two on each side of $\Pi^{\pm}$, are glued to form a crossing arc in $Y-L$. 
\item \emph{Ideal vertices} at the crossings of $\pi(L)$. These are all 4-valent.
\item Regions of $\pi(L)$ bounded by interior edges. These are called \emph{faces} of the chunk. They are not necessarily simply connected, but they are checkerboard colored.
\end{enumerate}
We further truncate ideal vertices. Because vertices are 4-valent, the resulting truncation turns each ideal vertex into a quad, called a \emph{truncation face}, bounded by \emph{truncation edges} (called boundary faces and boundary edges in \cite{HowiePurcell}). A weakly generalized alternating link complement admits such a decomposition~\cite[Propositions~3.1 and~3.3]{HowiePurcell}.

The key point is that the combinatorics of the chunk decomposition exactly matches the combinatorics of the link diagram; this comes from the fact that the diagram is alternating. Then the fact that the diagram is weakly prime and has bounded representativity restricts the way that surfaces can lie inside the chunk.

\subsection{Normal surfaces}
We now consider $(Z,\bdy Z)$ to be an essential surface embedded in $(Y-N(L),\bdy N(L))$, where the boundary of the surface $Z$ is possibly empty. We generally allow $Z$ to be orientable or nonorientable, unless otherwise stated. Recall a surface is \emph{essential} if it is incompressible, boundary incompressible, and is not boundary parallel.

\begin{assumption}
Throughout, we use the topological definition of incompressible surfaces: a surface $Z$ that is neither a disk nor a 2-sphere is incompressible in a 3-manifold $M$ if any disk $D$ with interior embedded in $M-N(Z)$, with boundary on $Z$, satisfies $\bdy D$ bounds a disk in $Z$. A surface that is not incompressible admits an essential compression disk, i.e.\ a disk $D$ with interior embedded in $M-N(Z)$ with $\bdy D$ an essential curve on $Z$. 

A 2-sphere is incompressible if and only if it does not bound a 3-ball. By convention, disks will be neither incompressible nor compressible.
\end{assumption}

We will count essential surfaces later in the paper in weakly generalized alternating links and their Dehn fillings. To do so, we put them into \emph{normal form} with respect to the chunk decompostion.
Normal form in this setting was introduced in Howie--Purcell~\cite[Definition~3.7]{HowiePurcell}; it generalizes the definition of \emph{standard position} for surfaces in classical alternating knots used by Menasco~\cite{Menasco1984} and Menasco--Thistlethwaite~\cite{MenascoThistlethwaite}, as well as \emph{normal form} for surfaces in Lackenby~\cite{lac00} and Futer--Gu{\'e}ritaud~\cite{fg09}. We will not need the full definition of normal form here, only the consequences of that definition from Howie--Purcell~\cite{HowiePurcell} and from Purcell--Tsvietkova~\cite{Paper1}. Therefore we refer to these references for definitions and examples, but we review the necessary results here.

In ~\cite[Theorem~3.8]{HowiePurcell}, it is shown that any essential surface in a 3-manifold with a chunk decomposition can be put into normal form with respect to the chunk decomposition.
When $Z$ is isotoped into normal form, it is cut into components $Z_i$, which are connected subsurfaces of $Z$ in normal form within a chunk. These may be closed or with boundary, and might have multiple boundary components. We will write $Z=\bigcup_i Z_i$.

Both for classical alternating links in $S^3$~\cite[Theorem~2]{Menasco1984} and for weakly generalized alternating links under certain conditions~\cite[Lemma~4.9]{HowiePurcell}, a closed surface $Z'$ admits meridianal compressions. After meridianal compressions the resulting surface $Z$ can be arranged to meet the diagram in meridians.

In~\cite[Theorem~6.2]{Paper1}, it is shown that an isotopy into normal form can be done such that boundary curves $\bdy Z$ that are meridians remain in \emph{meridianal form} after the isotopy, meaning that $\bdy Z$ cuts off exactly two corners of quads corresponding to truncation faces, one on $\Pi^+$ and one on $\Pi^-$. See \reffig{Meridian}, which is from~\cite{Paper1}. Moreover, such an isotopy does not increase weight, where the weight is the pair $(s(Z), t(Z))$, ordered lexicographically, where $s(Z)$ counts the number of intersections of $Z$ with interior edges, and $t(Z)$ counts the number of intersections with truncation edges.

\begin{figure}
%% Accompanies image file 'MeridinalForm.pdf' (pdf, eps, ps)
\begingroup%
  \makeatletter%
  \providecommand\color[2][]{%
    \errmessage{(Inkscape) Color is used for the text in Inkscape, but the package 'color.sty' is not loaded}%
    \renewcommand\color[2][]{}%
  }%
  \providecommand\transparent[1]{%
    \errmessage{(Inkscape) Transparency is used (non-zero) for the text in Inkscape, but the package 'transparent.sty' is not loaded}%
    \renewcommand\transparent[1]{}%
  }%
  \providecommand\rotatebox[2]{#2}%
  \newcommand*\fsize{\dimexpr\f@size pt\relax}%
  \newcommand*\lineheight[1]{\fontsize{\fsize}{#1\fsize}\selectfont}%
  \ifx\svgwidth\undefined%
    \setlength{\unitlength}{360bp}%
    \ifx\svgscale\undefined%
      \relax%
    \else%
      \setlength{\unitlength}{\unitlength * \real{\svgscale}}%
    \fi%
  \else%
    \setlength{\unitlength}{\svgwidth}%
  \fi%
  \global\let\svgwidth\undefined%
  \global\let\svgscale\undefined%
  \makeatother%
  \begin{picture}(1,0.292892)%
    \lineheight{1}%
    \setlength\tabcolsep{0pt}%
    \put(0,0){\includegraphics[width=\unitlength,page=1]{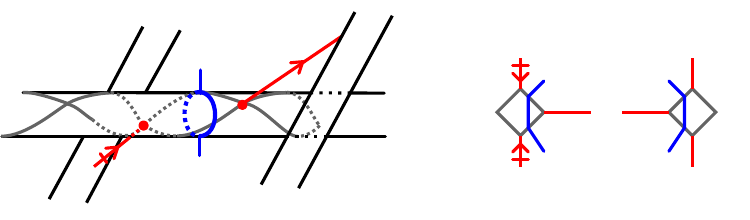}}%
    \put(0.75669826,0.16679205){\color[rgb]{0,0,0}\makebox(0,0)[lt]{\lineheight{1.25}\smash{\begin{tabular}[t]{l}$\Pi^+$\end{tabular}}}}%
    \put(0,0){\includegraphics[width=\unitlength,page=2]{MeridinalForm.pdf}}%
    \put(0.96561298,0.16568718){\color[rgb]{0,0,0}\makebox(0,0)[lt]{\lineheight{1.25}\smash{\begin{tabular}[t]{l}$\Pi^-$\end{tabular}}}}%
  \end{picture}%
\endgroup%

  \caption{Left: a surface with meridianal boundary can be isotoped to meet $N(L)$ transversely away from crossings. Right: in the chunk decomposition, $\bdy Z$ cuts off exactly two corners of truncation faces.}
  \label{Fig:Meridian}
\end{figure}

\begin{assumption}\label{MinimalComplexity}
From now on, when we put an essential surface in normal form, we always make two assumptions:
\begin{enumerate}
\item the surface is in meridianal form, and
\item of all ways to isotope the surface into normal, meridianal form, it has least weight.
\end{enumerate}
\end{assumption}

We obtain a labeling of curves of $Z_i$ that lie on the chunk boundary, described in~\cite[Section~5]{Paper1}, which we now review.
For a normal component $Z_j$ of the surface $Z$, each boundary component of $\bdy Z_j$ runs over truncation edges and interior edges of a chunk. Each intersection of $\bdy Z_j$ with an interior edge is labeled with an $S$ for ``saddle'', corresponding to an analogous labeling by Menasco~\cite{Menasco1984}; see also~\cite{HTT1, HTT2}. If a component of $\bdy Z$ is in meridianal form, there will be some $Z_i$ and $Z_j$ such that $\bdy Z_i$ and $\bdy Z_j$ each cut off exactly one corner of a truncation quad in meridianal form. Label these intersections with faces by $P$. More generally some $\bdy Z_k$ may run through truncation faces that may or may not be in meridianal form. Then $\bdy Z_k$ intersects one truncation edge on the way in, and one truncation edge on the way out; label each intersection with a $B$. The labeling of components of $\bdy Z_i$ by letters $S$, $B$, or $P$ associates a cyclic word to $\bdy Z_i$.

For classical alternating links, each $Z_i$ can be taken to be a disk lying in one of the topological 3-balls either above or below the link. The boundary curves $\bdy Z_i$ are determined by the associated words in $S$, $P$, and $B$ and by the position of each letter on the link diagram. Thus bounding the number of possible boundary words gives a bound on the number of surfaces; this is used in the work of Hass, Thompson, and Tsvietkova~\cite{HTT1, HTT2}. For weakly generalized alternating links, the subsurfaces $Z_i$ might have positive genus and multiple boundary components. Moreover, the subsurfaces are not in topological 3-balls anymore, but in chunks which might have complicated topology themselves. Hence we need more tools to restrict and control possible subsurfaces that arise. One such tool is the combinatorial area.

\subsection{Combinatorial area}\label{Subsec:Area}
Label each interior edge of a chunk with angle $\pi/2$, and label each truncation edge with angle $\pi/2$.

Write $Z=\bigcup_{j=1}^m Z_j$ where each $Z_j$ is a connected normal surface embedded in a chunk.
Each curve $\bdy Z_j$ will meet some number of interior edges, each labeled by $S$; denote the number by $(\#S)$. It will also meet truncation edges, with each either labeled by $B$ or with a pair in meridianal form labeled by a single $P$. Denote the number of instances of $B$ by $(\#B)$ and the number of instances of $P$ by $(\#P)$. On a single component of $N(L)$, all intersections of $\bdy Z_j$ with corresponding truncation edges will either all be labeled $B$ or all labeled $P$.

For this surface, the \emph{combinatorial area} of $Z_j$ is defined to be
\begin{equation}\label{Eqn:ZjArea}
  a(Z_j) = \frac{\pi}{2}(\#S) + \frac{\pi}{2}(\#B) + \pi(\#P) - 2\pi\chi(Z_j).
\end{equation}
The \emph{combinatorial area} of $Z$ then is
\begin{equation}\label{Eqn:ZArea}
  a(Z) = \sum_{i=1}^n a(Z_i).
\end{equation}

This construction satisfies a Gauss--Bonnet formula \cite[Proposition~3.12]{HowiePurcell}:
\begin{equation}\label{Eqn:GB}
a(S) = -2\pi\chi(Z).
\end{equation}

Furthermore, combinatorial area satisfies the following lemma.
\begin{lemma}[Lemma~8.5 of \cite{Paper1}]\label{Lem:Areas}
Let $Z_i$ be normal and connected with respect to a chunk. The combinatorial area of $Z_i$ satisfies the following.
\begin{enumerate}
\item[(1)]%% \label{Itm:ChiLeq-1}
  If $\chi(Z_i)< 0$ then $a(Z_i)\geq 2\pi$.
\item[(2)]%% \label{Itm:ChiGeq0}
  If $\chi(Z_i)\geq 0$ then either $a(Z_i)\geq \pi/2$ or $a(Z_i)=0$.
\item[(3)]%% \label{Itm:ZeroArea}
  Additionally, in the case that $a(Z_i)=0$, $Z_i$ is either:
  \begin{enumerate}
  \item[(a)] an essential torus or Klein bottle embedded in a chunk, hence $Z=Z_i$ is a torus or Klein bottle,
  \item[(b)] an annulus or a M\"obius band with boundary meeting no edges, or
  \item[(c)] a disk with $\bdy Z_i$ meeting exactly four edges of the chunk decomposition.
  \end{enumerate}
\end{enumerate}
\end{lemma}

Later, we will use combinatorial area to control the number of pieces $Z_i$ for a surface of a given genus. But this type of argument does not work for normal components of the surface that have zero area, and so we need to consider them separately. If we assume that $\Pi$ is chosen so that $Y- \Pi$ admits no essential embedded tori or Klein bottles, then there will be no zero area tori or Klein bottles in the chunk decomposition. If we require all regions of $\pi(L)$ on $\Pi$ to be disks, that is, the diagram is \emph{cellular}, then there will be no zero area annuli or M\"obius bands. We therefore focus on disks with zero area.

The only possible disks with zero area are enumerated in~\cite[Lemma~8.6]{Paper1}; they have boundaries labeled one of $PP$, $PSS$, $SSSS$, $BBBB$, $BBSS$, or in the case of links, $PBB$. The instance of $P$ in the final case will be relabeled with two instances of $B$; this allows us to identify such disks as $BBBB$ disks and use tools in that setting.

The zero area disks are considered in~\cite{Paper1}. By Theorem~9.1 in that paper, $SSSS$ disks only appear as an essential compression disk for $\Pi$, with boundary meeting the diagram $\pi(L)$ in exactly four interior edges. They will not occur if the representativity satisfies $r(\pi(L),\Pi)>4$. By \cite[Theorem~10.1]{Paper1}, there are no $PP$ disks (nor $PS$ nor $SS$ disks). By \cite[Theorem~10.2]{Paper1}, there are no $PSS$ disks.

If $Z_i$ comes from a spanning surface or more general surface with boundary and meets truncation edges, then there will be letters $B$. We cannot rule these out, but in~\cite[Section~11]{Paper1} it is shown that if such disks arise, they form larger disks together. In particular, the following appears in that paper.

\begin{theorem}[Theorem~11.4 of \cite{Paper1}]\label{Thm:NonBBBBdeterminesZ}
  Assume that the hat-representativity satisfies $\hat{r}(\pi(L), \Pi)>4$, and $\pi(L)$ is not a string of bigons on $\Pi$. Let $Z$ be an essential surface in normal form with respect to the chunk decomposition. Then all subsurfaces $Z_i$ that are neither $BBBB$ disks nor $BBSS$ disks, together with the link $L$, determine the surface $Z$ up to isotopy.
\end{theorem}

\begin{assumption}\label{CellularRep}
To rule out surfaces with zero area that are not disks, from now on we will assume that the diagram $\pi(L)$ is \emph{cellular}, i.e.\ all regions of $\Pi-\pi(L)$ are disks. Assume also that $Y-N(\Pi)$ admits no embedded essential tori or Klein bottles. To rule out $SSSS$ disks, we further assume $r(\pi(L),\Pi)>4$; note this implies that $\hat{r}(\pi(L),\Pi)>4$, so \refthm{NonBBBBdeterminesZ} also applies.

\end{assumption}

%%%%%%%%%%%%%%%%%%%%%%%%%%%%%%%%%%%%%%%%%%%%%%%%%%%%%%%%%%%%%%%%%
\section{Meridianal compressions}\label{Sec:MeridianCompressions}

Let $Z$ be a surface properly embedded in $Y-N(L)$. We say $Z$ is \emph{meridianally compressible} if there is an essential meridianal annulus $A$ with one component $\bdy_1 A$ of $\bdy A$ on $Z$, and the other component $\bdy_2 A$ of $\bdy A$ forming a meridian on $\bdy N(L)$. The requirement that $A$ be essential means it is not ambient isotopic rel $\bdy_1 A$ to an embedded annulus $A'$ in $Z$ with $\bdy A'$ the disjoint union of $\bdy_1 A$ and a meridian of $N(L)$. If there is no essential meridianal annulus for $Z$, we say it is \emph{meridianally incompressible}. For a meridianally compressible surface, performing surgery along $A$ yields a new surface with boundary forming a meridian of $N(L)$; this is called a \emph{meridianal compression} of $Z$.

As in \cite{HTT1}, we will count surfaces after first performing a maximal number of meridianal compressions. However, unlike \cite{HTT1}, we allow surfaces with additional boundary components and nonorientable surfaces. 
The next lemmas give bounds on the total number of meridianal compressions we must make, in terms of the Euler characteristic $\chi$ of the surface. 
Recall that the genus of the surface is the maximum number of disjoint simple closed curves that can be drawn on the surface without disconnecting it. For orientable surfaces, we denote the genus by $g_O$. For a nonorientable surface, the genus is equal to the number of cross-caps attached to a sphere; it is often called a nonorientable genus. Denote it by $g_N$. Then for a closed orientable surface $Z$, $\chi(Z)=2-2g_O$, and for a closed nonorientable one, $\chi(Z)=2-g_N$. For surfaces with $b$ boundary components, the Euler characteristic is $\chi(Z)=2-2g_O-b$ or $\chi(Z)=2-g_N-b$, for the orientable and nonorientable case respectively. We often denote genus by $g$, for both orientable and nonorientable surfaces, where this does not affect the calculation. 

\begin{lemma}\label{Lem:MeridCompressGeneral}
  Suppose $Z'$ is an essential surface with negative Euler characteristic, with or without boundary, orientable or nonorientable, that is properly embedded in $Y-N(L)$. Let $Z$ be obtained from $Z'$ by performing a maximal sequence of meridianal compressions. Then the Euler characteristic of $Z$ is $\chi(Z)=\chi(Z')$, and the meridianal compressions add at most $-4\chi(Z)+2$ boundary components. These are tubed in pairs to obtain $Z'$ from $Z$.

\end{lemma}

Note we could prove a better upper bound of $-2\chi(Z)$ when the surface is closed and orientable, by repeating an argument of \cite{HTT1}. However, since we only need a bound, and not a sharp bound, we state the most general bound to avoid multiple cases later.

\begin{proof}

A meridianal compression will be performed along a simple closed curve on the surface. We consider orientable and nonorientable surfaces, and for a nonorientable surface, such curve can be either one-sided or two-sided.

First, consider one-sided curves. Note that if an annulus $A$ has boundary $\bdy_1 A$ on a one-sided curve, then $A$ must wrap twice around the curve, and gluing $A$ to itself along $\bdy_1 A$ forms a M{\"o}bius band with meridianal boundary. 
However, \cite[Theorem~4.6]{HowiePurcell} implies that there is no embedded M\"obius band with meridianal boundary in a weakly generalized alternating link. Thus we may rule out one-sided curves for meridianal compressions.

Now consider two-sided curves, on orientable and nonorientable surfaces. The topological effect of a meridianal compression is to remove an open neighborhood of a (2-sided) curve $\bdy_1 A$ on $Z'$. This is an annulus, and the Euler characteristic of the surface before and after removing the interior of an annulus is unchanged. 

Note meridianal compression adds two boundary components to the surface. Thus to bound the number of added boundary components obtained by a maximal sequence of meridianal compressions, we must bound the number of such compressions. 

If $Z'$ is a closed orientable surface with negative Euler characteristic, it contains at most $3g-3$ disjoint nonparallel curves; these cut $Z'$ into pairs of pants. More generally, if $Z'$ is orientable with genus $g$ and $b$ boundary components, it contains $3g-3+b$ disjoint nonparallel curves that are not parallel to boundary components. The number of boundary components after removing a regular neighborhood of each of these curves is $6g-6 +2b \leq -3(2-2g-b)=-3\chi(Z')<-4\chi(Z')+2$. This is an upper bound for the number of boundary components produced by meridianal compression in this case.

If $Z'$ is a nonorientable surface with $b$ boundary components, it can be cut into planar surfaces and M\"obius bands along $g$ disjoint nonparallel two-sided curves. An additional $g+b-3$ curves cut the planar surfaces into pairs of pants, giving $2g-3+b$ disjoint nonparallel two-sided curves. The total number of boundary curves after removing regular neighborhoods of these curves is at most $4g-6+2b < 4g+4b - 8 + 2 = -4\chi(Z)+2$. 
\end{proof}

We also include one result from~\cite{Paper1} that allows us to rule out meridianal compressions for many surfaces with boundary on a link.

\begin{lemma}\label{Lem:MeridCompressRestrict}\cite[Proposition~7.1~(1)]{Paper1}
There is no meridianal compression of $Z$ to a component $L_i$ of $\bdy N(L)$ for which $\bdy Z \cap L_i$ is nonempty and non-meridianal.
\end{lemma}

\begin{assumption}\label{SurfaceAssumption}
From now on, assume $Z$ is essential, properly embedded in $Y-N(L)$. The surface $Z$ is not necessarily orientable, and possibly has boundary components, meridianal or nonmeridianal, that are always on $\bdy N(L)$. 
\end{assumption}

%%%%%%%%%%%%%%%%%%%%%%%%%%%%%%%%%%%%%%%%%%%%%%%%%%%%%%%%%%%%%%%%%
\section{Counting subsurfaces and their boundary curves}\label{Sec:Counting}

In this section, we begin our count of surfaces by bounding the total number of subsurfaces in a chunk making up a normal surface $Z$, and the total number of words associated with their boundary components.

\begin{lemma}\label{Lem:BddSurfaceCount}
Suppose $Z$ is meridianally incompressible. Let $Z=\cup_{i=1}^m Z_i$, where the $Z_i$ are normal subsurfaces in chunks. Then the number of subsurfaces $Z_i$ that are not disks with boundary $BBBB$ or $BBSS$ is at most $-4\chi(Z)$.
\end{lemma}

\begin{remark}
Recall that when we have a link with mixed boundary components, both meridianal and non-meridianal, we relabel all instances of $P$ with two instances of $B$. Thus in \reflem{BddSurfaceCount}, we view $PBB$ disks as $BBBB$ disks.
\end{remark}

\begin{proof}[Proof of \reflem{BddSurfaceCount}]
Let $n_1$ be the number of the $Z_i$ that have Euler characteristic $0$ or $1$, and are not disks with boundary $BBBB$ or $BBSS$. Let $n_2$ be the number with Euler characteristic at most $-1$. Then the total number of subsurfaces $Z_i$ considered in the lemma is $m=n_1+n_2$. By the Gauss--Bonnet theorem~\refeqn{GB}, by equation~\eqref{Eqn:ZArea}, and by \reflem{Areas}~(1),
\[
-2\pi\chi(Z) \: = \: a(Z) \:  = \: \sum_{i=1}^m a(Z_i) \: \geq \: \sum_{\chi(Z_i)=0,1} a(Z_i) + n_2\, 2\pi.
\]
For the subsurfaces with $\chi(Z_i)=1$ or $0$, we need to determine which have $a(Z_i)=0$. By \reflem{Areas}~(3), these can be tori, Klein bottles, annuli or M\"obius bands meeting no edges, or disks meeting four edges. Assumption~\ref{CellularRep}, that all regions of $\pi(L)$ on $\Pi$ are disks, rules out annuli and M\"obius bands meeting no edges: the boundary curve would lie in one of the cellular faces, and hence bound a disk there, which is not allowed for surfaces in normal form. Assumption~\ref{CellularRep} also requires that $Y-N(\Pi)$ admits no essential tori and Klein bottles, which rules out these surfaces. Thus if $a(Z_i)=0$ then $Z_i$ is a normal disk meeting four edges.

If $Z$ has only meridianal boundary components and $a(Z_i)=0$, then it is a disk with boundary of the form $PSS$, $PP$, or $SSSS$. But these are ruled out by \cite[Theorem~10.1]{Paper1}, \cite[Theorem~10.2]{Paper1}, and \cite[Theorem~9.1]{Paper1}, respectively, using the hypothesis that $r(\pi(L),\Pi)>4$ for the $SSSS$ case. If $Z$ has a non-meridianal boundary component, then we consider all intersections with truncation edges as labeled $B$. The only remaining normal disks with $a(Z_i)=0$ are $BBBB$ or $BBSS$ disks, ruled out by hypothesis.

Then \reflem{Areas}~(2) implies
\[ \sum_{\chi(Z_i)=0,1} a(Z_i) \geq n_1\, \frac{\pi}{2}, \quad \mbox{and hence}\]
\[ -2\pi\chi(Z) \geq n_1\,\frac{\pi}{2} + n_2\, 2\pi \geq \frac{\pi}{2}\, m.
\]
So $m\leq -4\chi(Z)$. Therefore, there are at most $-4\chi(Z)$ subsurfaces $Z_i$ that are not disks labeled $BBBB$ or $BBSS$.
\end{proof}

\begin{lemma}\label{Lem:WordAndLetterCount}
Let $Z$ be meridianally incompressible, properly embedded in $Y-N(L)$ in normal form. Suppose further that $Z$ satisfies one of the following cases.
\begin{enumerate}
\item[(1)] The surface $Z$ has only meridianal boundary.
\item[(2)] The surface $Z$ has at least one non-meridianal boundary component on $L$.
\end{enumerate}
Then there are at most $-10\chi(Z)$ boundary components on any $Z_i$ that
are not associated to $BBBB$ or $BBSS$ disks. Moreover, the number of intersections of such a boundary component with chunk edges is at most $-20\chi(Z)$.
\end{lemma}

\begin{proof}
Recall that $Z=\bigcup_{i=1}^m Z_i$. By \refeqn{ZArea}, and by Gauss--Bonnet \refeqn{GB}, the combinatorial area $a(Z)$ satisfies:
\[
-2\pi\chi(Z) \: = \: a(Z) \: = \: \sum_{i=1}^m a(Z_i).
\]
Each time a boundary component of $Z$ meets an interior or truncation edge of the chunk decomposition, there is a contribution of $\pi/2$ to the area. More precisely, if $w_i$ is the total number of intersections of all the boundary components of $Z_i$ with chunk edges, then
\begin{equation}\label{Eqn:Letters}
a(Z_i) \: = \: w_i\frac{\pi}{2} - 2\pi\chi(Z_i).
\end{equation}
Let $w$ be the number of intersections with chunk edges for all boundary components of $Z$ that are not associated to $BBBB$ or $BBSS$ disks.

In case~(1), when the surface has meridianal boundary, all words are in letters $P$ and $S$, so $w=\sum w_i$. In case~(2), when $Z_i$ is a $BBBB$ or $BBSS$ disk, we know its combinatorial area is $0$. Moreover, we know that $2\pi\chi(Z_i)\leq 0$ unless $Z_i$ is a disk. Thus from equation~\refeqn{Letters}, in both cases we have
\[
  a(Z) \: \geq \: w \frac{\pi}{2}
  - 2\pi\cdot \left|\{Z_i: Z_i \mbox{ is a disk, but not a $BBBB$ or $BBSS$ disk}\}\right|.
\]
The number of $Z_i$ that are disks but not $BBBB$ or $BBSS$ disks is at most the total number of such subsurfaces $Z_i$, which is at most $-4\chi(Z)$ by \reflem{BddSurfaceCount}. Thus
\[ -2\pi\chi(Z) \: \geq \: w\frac{\pi}{2} + 2\pi(4\chi(Z)), \]
and so $-20\chi(Z) \geq w$.

Any curve of $\partial Z_i$ that does not come from a $BBBB$ or $BBSS$ disk can contain at most $w$ intersections with chunk edges, so the total number of intersections is at most $w\leq -20\chi(Z)$. Any curve must have at least two intersections, since all regions of the chunks are disks and thus any closed curve $\bdy Z_i$ must enter and exit each region. Hence any $Z_i$ has at most $-10\chi(Z)$ boundary components.
\end{proof}

%%%%%%%%%%%%%%%%%%%%%%%%%%%%%%%%%%%%%%%%%%%%%%%%%%%%%%%%%%%%%%%%%
\section{Placing curves on the chunk boundary}\label{Sec:Intersections}

We now proceed with counting options for potential curves of intersection of $Z_i$ with the chunk boundary. Recall that for a fixed component $\Pi_j$ of the projection surface $\Pi$, there are two associated boundary components $\Pi_j^+$ or $\Pi_j^-$ for the chunk decomposition.  However, $\Pi^-$ is glued to $\Pi^+$, so once the curves are determined in each $\Pi_j^+$, they uniquely determine those on the $\Pi_j^-$. Thus we need only count the curves on $\Pi^+$.

We will count the curves by counting how they intersect the edges of the chunk decomposition: truncation edges and interior edges. 

\begin{definition}
An ordered subset of chunk edges on $\Pi_j^+$, up to cyclic order,  is called a \emph{combination}. 
\end{definition}

\begin{lemma}\label{Lem:CombinationsDetermineSubsfce}
A boundary component of $Z_i$ is determined up to isotopy on the chunk boundary by the respective combination.
\end{lemma}

\begin{proof}
Any two consecutive intersections in a combination will be connected by an arc of $\bdy Z_i$ in a face of a chunk. Because each face of a chunk is a disk, such an arc is unique up to isotopy. The only way the boundary component might not be unique is if multiple arcs with intersections on the same edges interleave in a face in different ways. But for a fixed combination, there will be a unique way to connect such arcs into a closed curve.
Therefore, a combination determines a closed curve on the chunk boundary. If we have two identical combinations, they will determine parallel curves on the chunk boundary. These are isotopic on the chunk boundary, as required.
\end{proof}

\begin{lemma}\label{Lem:CombinationsDetermineSurfaceBdry}
Once all combinations that correspond to normal subsurfaces $Z_i$ are fixed, then the slopes of boundary components of $Z$ on $\bdy N(L)$ are uniquely determined.  In particular, this is true even if we do not distinguish between $P$ and $B$. 
\end{lemma}

\begin{proof}
By \reflem{CombinationsDetermineSubsfce}, the boundary components of the subsurfaces $Z_i$ are determined up to isotopy by the fixed combinations. The subsurfaces will be identified across arcs in interior faces to form $Z$. This identification is via the homeomorphism that comes with the chunk decomposition of the link complement. The arcs of $Z_i$ on truncation faces become boundary components of $Z$. These glue uniquely at their endpoints, which lie on the boundary of interior faces, to give closed curves on $\bdy N(L)$.

Observe that combinations only include information on intersections with edges, and are independent of the letters $P$ and $B$, which we assign separately. Thus the slope is also determined independent of the assignment of $P$ or $B$.
\end{proof}

Counting combinations gives us the following lemma, which is similar to part of the argument of Hass, Thompson, and Tsvietkova~\cite{HTT1,HTT2}, despite the fact that here the 3-manifold is more general, and the normal subsurfaces $Z_i$ are general surfaces with boundary, while in \cite{HTT1,HTT2} analogous surfaces are disks. Also, unlike in \cite{HTT2}, a surface $Z$ with non-meridianal boundary is not necessarily a spanning surface for a link, i.e.\ its boundary curve does not have to have intersection number $1$ with a meridian.

\begin{lemma}\label{Lem:BoundaryCount}
Suppose $Z$ has boundary, either meridianal or non-meridianal on $\bdy N(L)$, and $\chi(Z)<0$, and $Z$ is meridianally incompressible.
Let $n$ be the number of crossings of the diagram $\pi(L)$.
Then the number of curves on the boundary of the chunk that could be a boundary component of some subsurface $Z_i$ on $\Pi^+$ is at most $C(n,\chi(Z))$, where:
\[
C(n,\chi(Z))= (6n)^{-20\chi(Z)}
\]
\end{lemma}

\begin{proof}
Let $\gamma$ be a boundary component of $\bdy Z_i$. By \reflem{CombinationsDetermineSubsfce}, such a curve is determined by its associated combination. 

On $\Pi_j^+$, there are $2n_j$ interior edges of $\Pi_j^+$ for each potential intersection with $\gamma$, where $n_j$ is the number of crossings on $\Pi_j$. Similarly, there are $4n_j$ truncation edges for each potential intersection. Thus for $\bdy Z_i$ on $\Pi_j^+$, there are at most $2n_j+4n_j \leq 6n_j$ choices for an intersection.
Summing over all possibilities for all $j$, there are at most $6n$ choices for an intersection on $\Pi^+ = \bigcup_j \Pi_j^+$.

By \reflem{WordAndLetterCount}, there are at most $-20\chi(Z)$ intersections with chunk edges for $\gamma$, as long as $\gamma$ is not the boundary of a $BBBB$ or $BBSS$ disk. But even if it is the boundary of a $BBBB$ or $BBSS$ disk, there are only $4<-20\chi(Z)$ intersections.
Thus there are at most $(6n)^{-20\chi(Z)}$ combinations of length $-20\chi(Z)$ on $\Pi^+$. However, we wish to count curves with \emph{at most} $-20\chi(Z)$ intersections, meaning we need to count ones with $-20\chi(Z)$ intersections, with $-20\chi(Z)-1$ intersections, etc., down to length $2$. For this, we make use of the observation
that our count of the combinations is already an overcount: For a  combination with $s<-20\chi(Z)$ intersections,
there will be a `last' intersection (depending on our choice of starting intersection). Define a combination of length $-20\chi(Z)$ associated to the shorter combination by requiring the combination to
hit the same last edge consecutively over and over. Such a combination does not correspond to a curve of $\partial Z_i$, but is counted in the bound $(6n)^{-20\chi(Z)}$. Therefore, $(6n)^{-20\chi(Z)}$ is an upper bound for the number of curves on $\Pi^{+}$.
\end{proof}

\begin{remark}
The upper bound in \reflem{BoundaryCount} is not sharp. Indeed, the number of possible combinations includes combinations that do not result in closed curves at all, and ones that result in curves that cannot correspond to boundary components of any $Z_i$ in normal position. However this bound will later allow us to obtain an upper bound on the number of surfaces that is polynomial in $n$. Observe that if $Z$ has many boundary components, or high genus, or even boundary components that intersect a meridian many times, then $\chi(Z)$ will be large as a consequence of the Gauss--Bonnet formula,  equation~\eqref{Eqn:GB}. Thus the upper bound will naturally be higher for such surfaces.
\end{remark}

%%%%%%%%%%%%%%%%%%%%%%%%%%%%%%%%%%%%%%%%%%%%%%%%%%%%%%%%%%%%%%%%%
\section{Surfaces in classical alternating link complements}\label{Sec:Classical}

In prior sections, we worked with a link $L$ projected onto a surface $\Pi$ in a general 3-manifold $Y$ satisfying a few mild hypotheses, and developed machinery for weakly generalized alternating links in $Y$. In this section, we restrict our attention to $Y=S^3$ and classical alternating links on $S^2$.

As noted in \refsec{WGA} above, a classical alternating link on $S^2$ in $S^3$ that is prime is also weakly prime, and has infinite representativity and hat-representativity. It is also checkerboard colored. Hence classical reduced, prime, alternating links are weakly generalized alternating links. Moreover, all the complementary regions in $S^2-\pi(L)$ are disks in this case, and hence the diagram is cellular. Thus it satisfies Assumptions~\ref{ManifoldProj} and~\ref{CellularRep}, hence satisfies all the hypotheses of the lemmas we have encountered so far.

A bound for the number of closed essential surfaces in classical alternating link complements was given in \cite{HTT1}; the same work gives the number of surfaces with meridianal boundary. The number of surfaces with non-meridianal boundary that are Seifert  (spanning, orientable) was given in \cite{HTT2}; this generalizes to nonorientable spanning surfaces. However the bound for the number of non-spanning surfaces with non-meridianal boundary was unknown. These surfaces have some boundary component that follows the knot along its longitude $q$ times, where $q>1$. We give the bound for the number of such surfaces in this section.

\begin{theorem}\label{Thm:MainCountSphere} Let $\pi(L)$ be a prime alternating projection of a link $L$ onto $S^2$ in $Y=S^3$, with $n$ crossings.
Let $Z$ be a fixed connected topological surface with Euler characteristic $\chi$. Then the number of ways, up to isotopy, that $Z$ can be properly embedded in $S^3-N(L)$ as an essential, meridianally incompressible surface is at most
\[
(6n)^{80\chi(Z)^2}.
\]
\end{theorem}

\begin{proof}
For a standard alternating projection of a link onto $S^2$ in $S^3$, there are exactly two chunks in its chunk decomposition, and these are both homeomorphic to balls. Each normal subsurface of $Z$ must be a disk, else it would be compressible within the ball, contradicting the fact that $Z$ is incompressible.

Suppose $Z_i$, $i=1, \dots, k$, are all normal subsurfaces of an essential surface $Z$ in chunks that are not $BBBB$ or $BBSS$ disks. By \reflem{BoundaryCount}, at most $C(n,\chi(Z))=(6n)^{-20\chi(Z)}$ curves on the boundary of the chunk could be a boundary component of $Z_i$. Because each $Z_i$ is a disk, $C(n,\chi(Z))$ therefore gives a bound on the number of possibilities for the disks $Z_i$.
Set aside all $Z_i$ that are $BBBB$ and $BBSS$ disks. The number of options for all remaining $\partial Z_i$ to be placed on the boundaries of the chunks is at most $C(n, \chi(Z))$. By \refthm{NonBBBBdeterminesZ}, these surfaces uniquely determine the surface $Z$. By \reflem{BddSurfaceCount}, there are at most $-4\chi(Z)$ such surfaces.
Therefore, $C(n, \chi(Z))^{-4\chi(Z)} = (6n)^{80\chi(Z)^2}$ is an upper bound for the number of such surfaces.
\end{proof}

\begin{corollary}\label{Cor:NumberOfSurfaces1} With the same hypotheses on $L$ and $\pi(L)$, the bound in \refthm{MainCountSphere} holds if one considers the number of isotopy classes of \emph{all} properly embedded connected essential and meridianally incompressible surfaces of fixed Euler characteristic $\chi$ instead of just one such fixed topological surface.
\end{corollary}

\begin{proof}
This follows directly from the proof of \refthm{MainCountSphere}. Note that we count all subsurfaces that might compose any surface with Euler characteristic $\chi$ in the proof. Since all such surfaces are made up of these subsurfaces, we obtain not just the number of embeddings of one topological surface, but rather the sum of the number of embeddings of all topological surfaces of Euler characteristic $\chi$.
\end{proof}

\begin{corollary}\label{Cor:QuasifuchsianSphere}
Let $\pi(L)$ be a prime alternating projection of a link $L$ onto $S^2$ in $Y=S^3$, with $n$ crossings, and suppose that $S^3-L$ is hyperbolic. 
Then for any fixed connected topological surface $Z$ with Euler characteristic $\chi$,  the number of ways, up to isotopy, that $Z$ can be properly embedded in $S^3-N(L)$ as a quasifuchsian surface is at most
$(6n)^{80\chi(Z)^2}$.
\end{corollary}
\begin{proof}
Essential surfaces in hyperbolic 3-manifolds are exactly one of the following three types (see~\cite{Bonahon}): quasifuchsian, a virtual fiber, contain accidental parabolics. Since a meridianal compression gives an accidental parabolic, any quasifuchsian surface satisfies the hypothesis of \refthm{MainCountSphere}. 
\end{proof}

Using similar techniques as in~\cite{HTT1}, we may almost immediately extend this to a count of essential surfaces, with or without boundary, that are not necessarily meridianally incompressible. The closed orientable surfaces counted in~\cite{HTT1} fall into this category, but so do many other surfaces not considered in~\cite{HTT1} or \cite{HTT2}. For example, we also consider non-spanning surfaces with a certain non-trivial slope, surfaces in link complements without a component of $\bdy Z$ on some of link components (such surfaces may meridianally compress), or surfaces with meridianal boundary on some but not all link components.

\begin{theorem}\label{Thm:MainCountSphereMeridCompress}
Let $\pi(L)$ be a prime alternating projection of a link $L$ onto $S^2$ in $Y=S^3$ with $n$ crossings. Let $Z$ be a connected topological surface with Euler characteristic $\chi(Z)$, with all boundary components (if any) on $N(L)$.
The number of ways, up to isotopy, that $Z$ can be properly embedded in $S^3-N(L)$ as an essential surface is at most
\begin{center}
  $(6n)^{80\chi(Z)^2} 2^{-4\chi(Z)+2}$
\end{center}
\end{theorem}

\begin{proof}
Perform a maximal number of meridianal compressions on $Z$. By \reflem{MeridCompressGeneral} this yields a surface $Z'$ with the same Euler characteristic as $Z$ and with at most $-4\chi(Z)+2$ additional boundary components. The number of such surfaces $Z'$ is governed by \refthm{MainCountSphere}: there are at most $(6n)^{80\chi(Z)^2}$ of these.

To obtain the original $Z$, we need to tube back together all boundary components created by meridianal compressions. Hence we need to multiply the upper bound by the number of ways to do these tubings. Mossessian studied ways to tube together surfaces in~\cite{Mossessian}. Although that paper is concerned with closed Heegaard surfaces, the argument in \cite[Lemma~3.7]{Mossessian} applies in more generality to construct tubed surfaces. It shows that a tubed surface is determined by any $-2\chi(Z)+1$ element subset of the $-4\chi(Z)+2$ boundary components of $Z'$ to be tubed. Thus there are at most ${-4\chi(Z)+2 \choose -2\chi(Z)+1}$ such tubings.
The final bound follows from the fact that ${n \choose k } \leq 2^n$ for $k\leq n$. 
\end{proof}

\begin{corollary}\label{Cor:NumberOfSurfaces2}
With the same hypotheses on $L$ and $\pi(L)$, the bound in \refthm{MainCountSphere} holds if one considers the number of isotopy classes of \emph{all} properly embedded connected essential surfaces of fixed Euler characteristic $\chi$, instead of just one such fixed surface.
\end{corollary}

\begin{proof}
Again this follows directly from the proof of \refthm{MainCountSphereMeridCompress} and \refcor{NumberOfSurfaces1}. 
\end{proof}

%%%%%%%%%%%%%%%%%%%%%%%%%%%%%%%%%%%%%%%%%%%%%%%%%%%%%%%%%%%%%%%%%
\section{The number of surfaces in arbitrary 3-manifolds}\label{Sec:SfcCount}

In \refthm{MainCountSphere}, each subsurface $Z_i$ is a disk in a chunk that is a ball, and therefore $Z_i$ is uniquely determined by its boundary. More generally, we will have subsurfaces $Z_i$ that can have varying genera. To count the total number of embedded non-isotopic surfaces in such cases, we need a bound on the number of incompressible surfaces embedded in the chunk $C$ with fixed boundary curves and with a fixed genus, both orientable and nonorientable. This will be denoted by $X$. Assuming we have such a bound, we obtain a bound on the number surfaces embedded in a weakly generalized alternating link exterior.

\begin{theorem}\label{Thm:MainCount}
Let $L$ be a weakly generalized alternating link in $Y$, with projection $\pi(L)$ with $n$ crossings on $\Pi$, where $Y$, $\Pi$ and $L$ satisfy Assumptions~\ref{ManifoldProj} and~\ref{CellularRep}. 
Suppose that for each 3-manifold component $\Sigma$ of $Y-N(\Pi)$
there is a universal bound $X$ on the number of isotopy classes of incompressible surfaces properly embedded in $\Sigma$ with fixed genus and fixed boundary curves on $\bdy N(\Pi)\cap \bdy\Sigma$.
Fix a topological surface $Z$, with genus $g$, Euler characteristic $\chi$, that is either orientable or non-orientable. Then the number of ways $(Z,\bdy Z)$ can be properly embedded in $(Y-N(L), \bdy N(L))$ as an essential, meridianally incompressible surface, up to isotopy, is at most:
\[ (X(g+1))^{-4\chi} \cdot (6n)^{-800\chi^3+80\chi^2} \]  
\end{theorem}

\begin{remark}
Note that $g$ in \refthm{MainCount} may refer to orientable or non-orientable genus, which have different relations with Euler characteristic $\chi$ as recalled in the beginning of \refsec{MeridianCompressions}. Once genus, Euler characteristic, and orientability type are fixed, as in \refthm{MainCount}, the expression for Euler characteristic in terms of $g$ is determined, as is the number of boundary components of the surface.
\end{remark}

\begin{remark}\label{Rem:Reformulations}
We do not include a result equivalent to Corollaries~\ref{Cor:NumberOfSurfaces1} and~\ref{Cor:NumberOfSurfaces2} in this section. The proofs of theorems in this section still construct (and therefore count) any surface of given genus, Euler characteristic and orientability type. But due to the classification of surfaces, once the genus, Euler characteristic and orientability type is fixed, the surface is unique topologically. We are only obtaining embeddings of this surface up to isotopy. We could have reformulated the above theorem as giving the number of non-isotopic embedded surfaces of fixed genus, Euler characteristic and orientability type.
\end{remark}

\begin{proof}[Proof of \refthm{MainCount}]
Recall that an embedded, meridianally incompressible essential surface $Z$ consists of subsurfaces $Z_i$ in chunks. %% As in \refthm{MainCountSphere},
Each $Z_i$ that is not a $BBBB$ or $BBSS$ disk may have from $0$ to $-10\chi(Z)$ boundary components by \reflem{WordAndLetterCount}. Each $Z_i$ may also have genus or nonorientable genus from $0$ to $g$. And for a fixed collection of boundary components and fixed $g$, $Z_i$ may be one of at most $X$ subsurfaces by assumption.

Suppose first the boundary of $Z_i$ consists of $w$ curves already placed on the boundary of chunks. We then have at most $X$ options for $Z_i$ of genus $g'$ inside its chunk. Since the genus of $Z_i$ may be from $0$ to $g$, we have
\[ G=\sum_{i=0}^g X = (g+1) X \]
options for a subsurface $Z_i$ with previously fixed boundary.

Now recall the options for $\bdy Z_i$ to be placed on faces of the chunks.
For $Z_i$ with 0 boundary components, we have one option, namely $Z_i=Z$.
For $Z_i$ with one boundary component, we have at most $C=C(n, \chi(Z))$ options by Lemma~\ref{Lem:BoundaryCount}. By \reflem{CombinationsDetermineSubsfce}, an option determines $\bdy Z_i$, and by \reflem{CombinationsDetermineSurfaceBdry}, all options determine $\bdy Z$.
For $Z_i$ with two boundary components, we have at most $C^2$ options by the same lemmas, etc. Hence the total number of options for the boundary of $Z_i$ is at most
\begin{align*}
  G\cdot 1+ & G\cdot C+G\cdot C^2+...+ G\cdot C^{-10\chi(Z)} \\
  & \leq G\cdot C^{-10\chi(Z)+1}\leq (g+1) X \cdot C^{-10\chi(Z)+1}=E
\end{align*}

By \refthm{NonBBBBdeterminesZ}, the surface $Z$ is determined by the subsurfaces $Z_i$ that are not $BBBB$ and $BBSS$ disks.
By \reflem{BddSurfaceCount}, there are at most $-4\chi(Z)$ such subsurfaces $Z_i$. Therefore the number of options for $Z$ is at most $E^{-4\chi(Z)}$.

If we expand out $E$ and $C$, and denote $\chi(Z)$ by  $\chi$, we obtain: 
\[(X(g+1))^{-4\chi} \cdot (6n)^{-800\chi^3+80\chi^2} \qedhere \]
\end{proof}

\begin{corollary}\label{Cor:QuasifuchsianGeneral}
With all the hypotheses of \refthm{MainCount}, suppose in addition that $Y-L$ admits a hyperbolic structure. Then the bound of \refthm{MainCount} bounds the number of ways, up to isotopy, that $Z$ can be properly embedded as a quasifuchsian surface.
\end{corollary}
\begin{proof}
As before, a quasifuchsian surface is meridianally incompressible. 
\end{proof}

\begin{theorem}\label{Thm:NotMeridIncompr} 
With all the hypotheses of \refthm{MainCount}, consider embedded essential surfaces that are no longer required to be meridianally incompressible. Then the number of such surfaces is at most:
\[
(2X(g+1))^{-4\chi+2} \cdot (6n)^{-800\chi^3+80\chi^2} \] 
\end{theorem}

\begin{proof}
For every 
surface $Z$, we first perform the maximal possible number of meridianal compressions. By \reflem{MeridCompressGeneral}, we obtain a new surface $Z'$ from $Z$ with $\chi(Z')=\chi(Z)$, and with at most $-4\chi(Z)+2$ new boundary components. The genus $g'$ of the new surface $Z'$ may have any value from $0, 1, 2, \dots, g$. Applying \refthm{MainCount}, we obtain the upper bound for the number of meridianally incompressible surfaces $Z'$ of genus $g'$: 
$V(g')=(X(g'+1))^{-4\chi} \cdot (6n)^{-800\chi^3+80\chi^2}$.

Since $g'$ may vary, we add options for different genera of meridianally incompressible surfaces, to obtain an upper bound for all of them: 
\[ V(0)+V(1)+V(2)+\dots +V(g)\leq (g+1)V(g). \]

To obtain the original closed surface $Z$ from $Z'$, we need to tube together the punctures introduced by meridian compression. Hence we need to multiply the upper bound for the number of meridianally incompressible surfaces by the number of ways to tube its meridianal punctures. As in the proof of \refthm{MainCountSphereMeridCompress}, by work of Mossessian~\cite[Lemma~3.7]{Mossessian}, the number of tubings producing non-isotopic embedded surfaces is: 
\[ {-4\chi(Z)+2 \choose -2\chi(Z)+1} \]

Hence, using the fact that ${n \choose k}\leq 2^n$, the bound is:
\begin{align*}
 X^{-4\chi} & \cdot (g+1)^{-4\chi+1} \cdot (6n)^{-800\chi^3+80\chi^2}{-4\chi+2 \choose -2\chi+1} \\
&\leq X^{-4\chi}\cdot (g+1)^{-4\chi+1} \cdot (6n)^{-800\chi^3+80\chi^2}\cdot 2^{-4\chi+2} \\
&\leq (2X(g+1))^{-4\chi+2} \cdot (6n)^{-800\chi^3+80\chi^2} \qedhere
\end{align*}
\end{proof}

%%%%%%%%%%%%%%%%%%%%%%%%%%%%%%%%%%%%%%%%%%%%%%%%%%%%%%%%%%%%%%%%%
\section{Bounds for isotopy classes of subsurfaces}\label{Sec:Tori}

To apply \refthm{MainCount}, we need to know bounds on $X$, the number of isotopy classes of incompressible surfaces properly embedded in a chunk with fixed genus and fixed boundary. Recall that chunks are submanifolds of $Y-N(L)$, bounded by connected components of the projection surface $\Pi$. For classical alternating links in $S^3$, there are just two chunks, above and below the projection plane, and topologically each one is a 3-ball. The only surfaces in a 3-ball are disks, each uniquely determined up to isotopy by its boundary, i.e.\ here $X=1$. This leads to the results in \refsec{Classical} for classical alternating links.

Surprisingly, there is not much in the literature giving such bounds in more general cases. 
When the surface is incompressible and also boundary incompressible, one can bound the number of isotopy classes of essential surfaces in compact 3-manifolds with boundary that are irreducible, boundary irreducible, anannular and atoroidal using techniques from classical normal surface theory; see for example Matveev~\cite{Matveev}, Jaco and Oertel~\cite{JacoOertel}, and Hass, Lagarias and Pippenger~\cite{HLP}. However, for our surfaces in a chunk, we cannot assume boundary incompressibility.

When we consider simultaneously embedded surfaces, i.e.\ the number of disjoint non-parallel incompressible surfaces properly embedded in a 3-manifold, results are known even in the case that the surface is not boundary incompressible. B.~Freedman and M.~Freedman gave a bound on the number of simultaneously embedded surfaces with bounded Euler characteristic~\cite{FreedmanFreedman}. However again this is not sufficient for our purposes; we need to count more than just the surfaces that can be simultaneously embedded.
Such a count seems only to be known for a few classes of compact 3-manifolds with boundary besides balls. We treat two cases here.

\subsection{Surfaces in thickened tori and solid tori}
The following is stated by Przytycki~\cite[Theorem~2.3 and Corollary~2.5]{Przytycki}.

\begin{theorem}\label{Thm:ThickenedTorusSfces}
Let $F$ be a properly embedded incompressible surface in $T^2\times I$ that is not a boundary parallel disk. Then $F$ is isotopic to either
\begin{enumerate}
\item an annulus $(\gamma)\times I$, for $\gamma\subset T^2$ a nontrivial simple closed curve, or
\item an annulus or torus parallel to the boundary, or
\item a nonorientable surface that is uniquely determined by two different slopes, one on $T^2\times\{0\}$ and one on $T^2\times\{1\}$.
\end{enumerate}
\end{theorem}

\begin{remark}
It is stated in \cite{Przytycki} that these results follow from work of Bredon and Wood~\cite{BredonWood} and Rubinstein~\cite{Rubinstein:OneSided} on surfaces in lens spaces. While indeed similar proof techniques can be used, Theorem~\ref{Thm:ThickenedTorusSfces} is not stated in that form in those papers. Similarly, more recent unpublished work of Bartolini~\cite{Bartolini:OneSided} gives an alternative proof of uniqueness of the surfaces in \refthm{ThickenedTorusSfces}~(3). However again the fact we need, that all such surfaces in these 3-manifolds have this form, is not stated in that paper. Thus we outline the proof in the Appendix (Section~\ref{Sec:Appendix}) for completeness.
\end{remark}

\begin{corollary}\label{Cor:SolidTorusSfces}
Each incompressible surface that is not parallel to the boundary in a solid torus $S^1\times D^2$ is determined up to isotopy by a simple closed curve on the boundary torus.
\end{corollary}

\begin{proof}
We may isotope the surface $F$ so that it meets the core of the solid torus transversely in a finite number of points; equivalently, $F$ meets a regular neighbhourhood $V$ of the core in a finite number of meridian disks. Then $G=F\cap (S^1\times D^2-V)$ is an incompressible surface, possibly with multiple components, in a thickened torus. By \refthm{ThickenedTorusSfces}, each component of $G$ is a boundary parallel disk or has one of three forms. 

Because $F$ meets $V$ in meridian disks, no component of $G$ will be a boundary parallel disk with boundary on $\bdy V$. If there is a component $G_i$ of $G$ that is an annulus parallel to the boundary $\bdy V$, then isotoping the surface $G_i$ along a boundary compressing disk into $V$ joins two disks of $F\cap V$ with a strip, creating a boundary parallel disk in $V$ that can be isotoped outside of $V$, reducing the number of components of $V\cap F$. Repeating a finite number of times, we may assume that one of three things holds: (a)~$G$ is a vertical annulus with the form of~(1) in \refthm{ThickenedTorusSfces}; (b)~$F\cap V$ is empty, in which case $F$ is boundary parallel in the solid torus (and also in $T^2\times I$); (c)~$G$ has the form of~(3) in \refthm{ThickenedTorusSfces}, hence $F\cap V$ is a single meridian disk.  In case~(a), $F$ is a compressing disk for $S^1\times D^2$. In case~(c), \refthm{ThickenedTorusSfces} implies that $F\cap (S^1\times D^2-V)$ is nonorientable, and 
this surface is uniquely determined by its boundary components on $F\cap \bdy V$ and on $F\cap (S^1\times S^1)$. Since the boundary component on $F\cap \bdy V$ is a meridian, the surface is determined by $F\cap (S^1\times S^1)$.
\end{proof}

The following result is more concrete than Theorems~\ref{Thm:NotMeridIncompr} and~\ref{Thm:MainCount}, since it does not depend on $X$. Remark~\ref{Rem:Reformulations} still applies here, i.e.\ the following corollary gives the bound on the number of embedded non-isotopic surfaces with fixed genus (orientable or non-orientable), Euler characteristic and orientability type. %Topologically they are all embeddings of the same surface.

\begin{corollary}\label{Cor:TorusChunks}
Suppose that $\Pi$ is a torus in the 3-manifold $Y$ that is either:
\begin{enumerate}
\item[(a)] $\Pi$ is the Heegaard torus in a lens space $Y$, or
\item[(b)] $\Pi$ is the torus $T^2\times\{0\}$ in the thickened torus $Y=T^2\times[-1,1]$, or
\item[(c)] $\Pi$ is a boundary parallel torus within a solid torus $Y$.
\end{enumerate}
Let $L$ be a weakly generalized alternating link in $Y$ with projection $\pi(L)$ on $\Pi$, with a cellular diagram and in cases~(a) and~(c), representativity $r(\pi(L),\Pi)>4$. Let $n$ be the number of crossings in the diagram $\pi(L)$. Let $Z$ be a fixed connected topological surface with Euler characteristic $\chi$. Then up to isotopy, the number of ways that $(Z, \bdy Z)$ can be properly embedded in $(Y-N(L), \bdy N(L))$ as an essential, meridianally incompressible surface is at most:
\[  (2(g+1))^{-4\chi} \cdot (6n)^{-800\chi^3+80\chi^2} \]
The number of ways to embed it as an essential surface with no restriction on meridianal compressibility is at most:
\[  (4(g+1))^{-4\chi+2} \cdot (6n)^{-800\chi^3+80\chi^2}
\]
\end{corollary}

\begin{proof}
Observe that $Y$, $\Pi$, and $L$ above satisfy Assumptions~\ref{ManifoldProj} and~\ref{CellularRep}, so \refthm{MainCount} and \refthm{NotMeridIncompr} apply. 
In all cases, the chunks consist of solid tori or thickened tori. 
When we isotope a properly embedded essential  surface homeomorphic to $Z$ into normal form, it will have boundary only on the torus corresponding to $\bdy N(\Pi)$. Hence \refthm{ThickenedTorusSfces} (in cases (b) and (c)) and \refcor{SolidTorusSfces} (in case (a) and (c)) apply. Therefore genus and boundary curves on $\bdy N(\Pi)$ uniquely determine the surface up to isotopy if $Z$ is nonorientable. Theorem~\ref{Thm:ThickenedTorusSfces} also allows $Z$ to be a boundary parallel annulus, i.e.\ $Z$ is then parallel to the torus $\Pi$. There are at most two different boundary parallel annuli with the same boundary components.  Thus for all these manifolds $X\leq 2$, and we may set $X=2$ in \refthm{MainCount} and \refthm{NotMeridIncompr}.
\end{proof}

\subsection{Orientable surfaces in thickened surfaces}\label{Sec:Virtual}
When we restrict to orientable surfaces, more results are known. For example, the following theorem appears in a paper of Waldhausen from 1968~\cite[Corollary~3.2]{Waldhausen}. 

\begin{theorem}[Waldhausen]\label{Thm:Waldhausen}
Suppose $G$ is a properly embedded orientable incompressible surface in the 3-manifold $F\times I$, for $F$ an orientable surface that is not the 2-sphere, and suppose $\bdy G$ is contained in $F\times\{1\}$. Then every component of $G$ is boundary parallel, parallel to a surface in $F\times\{1\}$.
\qed
\end{theorem}

Theorem~\ref{Thm:Waldhausen} allows us to apply a bound to cellular weakly generalized alternating links with projection surface $\Pi=F\times\{0\}$ in a thickened surface $Y=F\times[-1,1]$. These are \emph{virtual alternating links}. Virtual links were introduced by Kauffman~\cite{Kauffman:VirtualKnots}, and subsequently shown to be equivalent to diagrams of links in thickened surfaces up to moves called stabilisation and destabilisation. Kuperberg proved there exists a unique minimal genus surface $F$ for which the link embeds in $F\times[-1,1]$ with diagram on $F\times\{0\}$ admitting no destabilisations~\cite{Kuperberg:Virtual}. Such a diagram must be cellular. If we require the diagram to be weakly prime and alternating, it is exactly a weakly generalized alternating link, and satisfies all assumptions required for our theorems.

\begin{corollary}\label{Cor:CountVirtualKnots} 
Let $\pi(L)$ be a virtual link $L$ with a weakly prime alternating projection onto a surface $\Pi=F\times\{0\} \subset Y=F\times[-1,1]$ with $n$ crossings. Fix a connected, orientable topological surface $Z$ with Euler characteristic $\chi$. Then up to isotopy, the number of ways to properly embed $(Z, \bdy Z)$ as an essential, meridianally incompressible surface in $(Y-N(L), \bdy N(L))$  is at most:
\[ (2(g+1))^{-4\chi} \cdot (6n)^{-800\chi^3+80\chi^2} 
\]
If in addition $Y-L$ admits a hyperbolic structure, then the above bounds quasifuchsian surfaces.

If we do not require meridianal incompressibility, the number of ways is at most:
\[ (4(g+1))^{-4\chi(Z)+2}(6n)^{-800\chi^3+80\chi^2} \]
\end{corollary}

\begin{proof}
Note first that $L$ is a weakly generalized alternating link, with infinite representativity because $\Pi$ has no compression discs in $Y$. Note also that $Y$, $\Pi$, and $L$ satisfy Assumptions~\ref{ManifoldProj} and~\ref{CellularRep}. Thus \refthm{MainCount} and \refthm{NotMeridIncompr} will apply, once we determine a value for $X$.
The chunk decomposition consists of two chunks of the form $F\times I$.
Any embedded essential surface in $Y-N(L)$ with boundary on $N(L)$, when isotoped into normal form, will have boundary on only one side of the chunk, namely the boundary component of $F\times I$ that meets the diagram $\pi(L)$. Then \refthm{Waldhausen} implies that such a surface is boundary parallel when orientable, hence uniquely determined by its boundary curves, with the surface lying on at most two sides of a component of the boundary curve. For these manifolds, we may set  $X=2$.
\end{proof}

\begin{question}\label{Qn} 
Consider a 3-manifold $\Sigma$ with a fixed set $S$ of curves on $\bdy\Sigma$. Suppose $Z$ is an incompressible (but not neccessarily boundary incompressible) surface of genus $g$, and $\bdy Z=S$. For which families of 3-manifolds $\Sigma$ does there exists an upper bound on the number of isotopy classes of embeddings of $Z$  that is a constant or depends only on $g$?
\end{question}

%%%%%%%%%%%%%%%%%%%%%%%%%%%%%%%%%%%%%%%%%%%%%%%%%%%%%%%%%%%%%%%%%
\section{Dehn fillings}\label{Sec:DehnFillings}

For manifolds obtained by Dehn filling a link $L$ that satisfies the hypotheses of \refthm{MainCount} and \refthm{NotMeridIncompr}, we also obtain a bound on the number of embedded surfaces. A result concerning fillings of classical alternating links in the 3-sphere, but excluding finitely many filling slopes, is given in \cite{HTT1}. In this section,  both the ambient 3-manifold and the set of slopes are more general.

\begin{theorem}\label{Thm:DehnFilling}
Let $\pi(L)$ be a weakly generalized alternating projection of a knot
$L$ onto a projection surface $\Pi$ in a compact, irreducible, orientable 3-manifold $Y$, satisfying Assumptions~\ref{ManifoldProj} and~\ref{CellularRep}.
Let the number of crossings of $\pi(L)$ be $n$. 
Suppose that there is a universal bound $X$ on the number of isotopy classes of incompressible surfaces properly embedded in $\Sigma$ with fixed genus and a fixed finite set of boundary curves on $\bdy N(\Pi)\cap \bdy\Sigma$.
Finally, suppose $Y-L$ is hyperbolic. Take $(p,q)$ to be a nonmeridianal slope of length $\len(p,q)>2\pi$, and let $Q=4\pi(g-1)/(\len(p,q)-2\pi)$.
Then the number of isotopy classes of closed orientable genus $g$ surfaces in the manifold obtained by $(p,q)$ Dehn filling on $Y-L$ is at most
\begin{align}\label{Eqn:DehnFillingBound}  
& (2X(g+1))^{-4\chi+2} (6n)^{-800\chi^3+80\chi^2} + \\
	& \hspace{1in} \sum_{b=1, ..., Q} (2X(g+1))^{8g+4b-6} (6n)^{-800(2g+b-2)^3+80(2g+b-2)^2} \notag
\end{align}

In particular, the bound is polynomial in $n$.
\end{theorem}

\begin{proof}
Hass, Rubinstein, and Wang showed in \cite[Theorem~4.1]{HassRubinsteinWang} that if $Y-L$ is hyperbolic, then for all but finitely many Dehn filling slopes, the essential surfaces of genus $g$ in the Dehn filling of $Y-L$ along that slope are exactly the essential surfaces in $Y-L$. Thus for these slopes, the upper bound comes from \refthm{NotMeridIncompr}. In general, for all slopes, this is the first term in the sum.

Also in general, for all slopes, if a surface is essential after Dehn filling, then either it was essential before Dehn filling, in which case it will be counted in the first term of the given sum or before Dehn filling it was a genus $g$ essential surface with some number of boundary components, each with slope $(p,q)$ on $Y-L$. Let $b$ be the number of such boundary components. Then the number of essential surfaces with genus $g$ and $b$ boundary components  is counted by \refthm{NotMeridIncompr}. For fixed $b$, there are at most 
\[ (2X(g+1))^{4(2g+b-2)+2}\cdot (6n)^{-800(2g+b-2)^3+80(2g+b-2)^2} \]
such surfaces. These are the terms in the second part of the sum, summed over all values of $b$. It remains to bound $b$.

Let $S$ be an essential surface of genus $g$ embedded in $Y-N(L)$ with $b$ boundary components, each of slope $(p,q)$ on $\bdy N(L)$.
We may give $S$ a pleating (see, for example \cite[Proposition~8.40]{Purcell:HKT}); $S$ then inherits a complete hyperbolic metric from the pleating in $Y-N(L)$. Note that the area of $S$ within the cusp of $Y-L$ is at least equal to $b$ times the length of the slope $(p,q)$ (e.g.~\cite[Lemma~8.44]{Purcell:HKT}). Thus $\area(S) \geq b\len(p,q)$.

By Gauss--Bonnet, we also have $\area(S) = -2\pi\chi(S) = 2\pi(2g+b-2)$.
It follows that
\[ b\len(p,q) -2b\pi \leq 4\pi(g-1), \quad \mbox{or} \quad
b \leq \frac{4\pi(g-1)}{\len(p,q)-2\pi } =Q. \]
Thus summing over all possible integers $b$ in this range gives the result.
\end{proof}

If we adjust our hypotheses on the link and the filling slope slightly, $Y-L$ will automatically be hyperbolic with a bound on the slope length.

\begin{corollary}\label{Cor:DehnFilling}
Let $\pi(L)$ be a weakly generalized alternating projection of a knot
$L$ onto a projection surface $\Pi$, with $n$ crossings, in a 3-manifold $Y$, satisfying Assumptions~\ref{ManifoldProj} and~\ref{CellularRep}. Suppose that $Y-N(\Pi)$ contains no essential annuli with both boundary components on $\bdy Y$. 
Suppose there is a universal bound $X$ on the number of isotopy classes of incompressible surfaces properly embedded in each component of $Y-N(\Pi)$ with fixed genus and fixed boundary.
Let $\sigma=(p, q)$ be a slope on $\bdy N(K)$, with $|q| > 5.627(1-\chi(\Pi)/n)$.
Then the number of isotopy classes of closed orientable genus $g$ surfaces in the manifold obtained by $(p,q)$ Dehn filling of $Y-L$ is bounded above by a polynomial function of $n$, equal to the bound of \eqref{Eqn:DehnFillingBound}. 
\end{corollary}

\begin{proof}
Under the given assumptions, $Y-L$ is hyperbolic by \cite[Theorem~1.1]{HowiePurcell}. 
It is shown in the proof of \cite[Corollary~7.2]{HowiePurcell} that the length of $\sigma$ is at least
 \[ \len(\sigma) \geq \frac{3.35n|q|}{3(n-\chi(\Pi))}. \]
This is at least $2\pi$ under the given assumption on $q$, so  \refthm{DehnFilling} holds. 
\end{proof}

\section{Appendix: Proof of \refthm{ThickenedTorusSfces}}\label{Sec:Appendix}

\begin{proof}[Proof of \refthm{ThickenedTorusSfces}]

First suppose $F$ is both incompressible and boundary incompressible. If it is orientable, then it must be vertical or horizontal; see for example \cite[Theorem~VI.34]{Jaco:3-mfld}. In $T^2\times I$, a vertical surface is of the form $\{\gamma\}\times I$, and a horizontal surface is a torus parallel to the boundary. A non-orientable surface will be pseudo-horizontal or pseudo-vertical in general, as defined and proved by Frohman~\cite{Frohman}. However, in $T^2\times I$, both of these definitions reduce to the above notion of horizontal and vertical. 
  
Now assume $F$ is boundary compressible. 
If a component of $\bdy F$ bounds a disk $D$ on $T^2\times\{0\}$ or $T^2\times\{1\}$, then because $F$ is incompressible, $\bdy D$ pushed slightly into $F$ bounds a disk in $F$; it follows that $F$ is a boundary parallel disk, contradicting the hypothesis. So we may suppose that each component of $\bdy F$ is essential.

Consider an essential boundary compression disk $D$ for $F$. The boundary $\bdy D$ consists of two arcs $\bdy D = \alpha \cup \beta$, with $\alpha\subset F$ and $\beta \subset T^2\times\{0\}$ or $T^2\times\{1\}$; without loss of generality, say $T^2\times\{0\}$. Suppose first that the endpoints of $\alpha$ lie on distinct components $\gamma_1$ and $\gamma_2$ of $\bdy F$. Then $(T^2\times\{0\}) - (\gamma_1\cup \gamma_2)$ consists of two annuli, one containing $\beta$, with $\beta$ being an essential arc on the annulus.
This is depicted in \reffig{ThickenedTorusSfces1}, on the left.
Then $(T^2\times\{0\}) - N(\beta \cup \gamma_1 \cup \gamma_2)$ has a disk component $E$ in $T^2\times\{0\}$. The disk is depicted in \reffig{ThickenedTorusSfces1}, on the right.
Consider the arcs on $F$ given by $\bdy N(\alpha)\cap F - N(\bdy \alpha)$. Take the union of these with arcs $\gamma_1\cup \gamma_2 -N(\bdy\alpha)$ on $T^2$. The resulting closed curve bounds a disk $E \cup \bdy N(D)$ in $T^2\times I$, also depicted in \reffig{ThickenedTorusSfces1}, on the right.
The disk can be isotoped to be parallel to a disk in $T^2\times\{0\}$. Incompressibility of $F$ implies that it must bound a disk in $F$ as well. It follows that $F$ is a boundary parallel annulus.

\begin{figure}
\centering
  %% Creator: Inkscape 1.2.1 (9c6d41e410, 2022-07-14), www.inkscape.org
%% Accompanies image file 'ThickenedTorusSfces1.pdf' (pdf, eps, ps)
\begingroup%
  \makeatletter%
  \providecommand\color[2][]{%
    \errmessage{(Inkscape) Color is used for the text in Inkscape, but the package 'color.sty' is not loaded}%
    \renewcommand\color[2][]{}%
  }%
  \providecommand\transparent[1]{%
    \errmessage{(Inkscape) Transparency is used (non-zero) for the text in Inkscape, but the package 'transparent.sty' is not loaded}%
    \renewcommand\transparent[1]{}%
  }%
  \providecommand\rotatebox[2]{#2}%
  \newcommand*\fsize{\dimexpr\f@size pt\relax}%
  \newcommand*\lineheight[1]{\fontsize{\fsize}{#1\fsize}\selectfont}%
  \ifx\svgwidth\undefined%
    \setlength{\unitlength}{278.59177208bp}%
    \ifx\svgscale\undefined%
      \relax%
    \else%
      \setlength{\unitlength}{\unitlength * \real{\svgscale}}%
    \fi%
  \else%
    \setlength{\unitlength}{\svgwidth}%
  \fi%
  \global\let\svgwidth\undefined%
  \global\let\svgscale\undefined%
  \makeatother%
  \begin{picture}(1,0.25151967)%
    \lineheight{1}%
    \setlength\tabcolsep{0pt}%
    \put(0,0){\includegraphics[width=\unitlength,page=1]{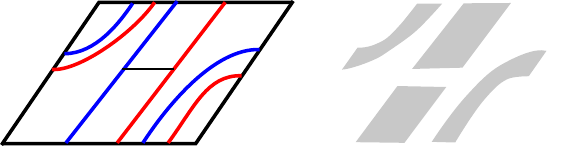}}%
    \put(0.22962752,0.09806901){\color[rgb]{0,0,0}\makebox(0,0)[lt]{\lineheight{1.25}\smash{\begin{tabular}[t]{l}$\beta$\end{tabular}}}}%
    \put(0,0){\includegraphics[width=\unitlength,page=2]{ThickenedTorusSfces1.pdf}}%
    \put(0.29097634,0.17714707){\color[rgb]{0,0.58823529,0.19607843}\makebox(0,0)[lt]{\lineheight{1.25}\smash{\begin{tabular}[t]{l}$\alpha$\end{tabular}}}}%
    \put(0,0){\includegraphics[width=\unitlength,page=3]{ThickenedTorusSfces1.pdf}}%
    \put(0.78929647,0.19964142){\color[rgb]{0,0,0}\makebox(0,0)[lt]{\lineheight{1.25}\smash{\begin{tabular}[t]{l}$E$\end{tabular}}}}%
    \put(0.24520888,0.13983656){\color[rgb]{0,0,0}\makebox(0,0)[lt]{\lineheight{1.25}\smash{\begin{tabular}[t]{l}$D$\end{tabular}}}}%
    \put(0,0){\includegraphics[width=\unitlength,page=4]{ThickenedTorusSfces1.pdf}}%
  \end{picture}%
\endgroup%

\caption{Left: The disk $D$ meeting distinct curves on $T^2\times\{0\}$. Right: The disk $E$ and the disk $D\cup E$.}
\label{Fig:ThickenedTorusSfces1}
\end{figure}

So suppose that the arc $\alpha$ has both endpoints on the same curve of $\bdy F$. The arc $\beta$ cannot cut off a disk $E$ in $T^2\times\{0\}-\bdy F$, else $E\cup D$ would be a compression disk for $F$. Thus $\beta$ is an essential arc in an annulus $T^2\times\{0\}-\bdy F$. It follows that $\bdy F$ has only one component on $T^2\times \{0\}$. Thus an arc $\gamma$ of $\bdy F$ runs from one endpoint of $\bdy \alpha$ to the other.
Because $\gamma$ is an essential curve on the torus, it runs from one side of $D$ to the opposite side, as in \reffig{ThickenedTorusSfces3}, left. It cannot run to the same side of $D$, else $\gamma\cup \beta$ would bound a disk $E$ on $T^2\times \{0\}$, and $E\cup D$ would be a compressing disk for $F$.

\begin{figure}
  %% Creator: Inkscape 1.2.1 (9c6d41e410, 2022-07-14), www.inkscape.org
%% Accompanies image file 'ThickenedTorusSfces3.pdf' (pdf, eps, ps)
\begingroup%
  \makeatletter%
  \providecommand\color[2][]{%
    \errmessage{(Inkscape) Color is used for the text in Inkscape, but the package 'color.sty' is not loaded}%
    \renewcommand\color[2][]{}%
  }%
  \providecommand\transparent[1]{%
    \errmessage{(Inkscape) Transparency is used (non-zero) for the text in Inkscape, but the package 'transparent.sty' is not loaded}%
    \renewcommand\transparent[1]{}%
  }%
  \providecommand\rotatebox[2]{#2}%
  \newcommand*\fsize{\dimexpr\f@size pt\relax}%
  \newcommand*\lineheight[1]{\fontsize{\fsize}{#1\fsize}\selectfont}%
  \ifx\svgwidth\undefined%
    \setlength{\unitlength}{278.59177208bp}%
    \ifx\svgscale\undefined%
      \relax%
    \else%
      \setlength{\unitlength}{\unitlength * \real{\svgscale}}%
    \fi%
  \else%
    \setlength{\unitlength}{\svgwidth}%
  \fi%
  \global\let\svgwidth\undefined%
  \global\let\svgscale\undefined%
  \makeatother%
  \begin{picture}(1,0.25151967)%
    \lineheight{1}%
    \setlength\tabcolsep{0pt}%
    \put(0,0){\includegraphics[width=\unitlength,page=1]{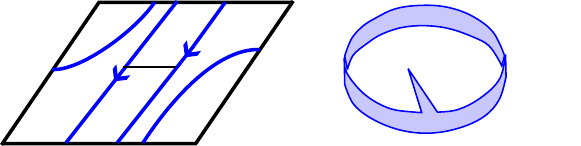}}%
    \put(0.22962752,0.09806901){\color[rgb]{0,0,0}\makebox(0,0)[lt]{\lineheight{1.25}\smash{\begin{tabular}[t]{l}$\beta$\end{tabular}}}}%
    \put(0.24520888,0.14522079){\color[rgb]{0,0,0}\makebox(0,0)[lt]{\lineheight{1.25}\smash{\begin{tabular}[t]{l}$D$\end{tabular}}}}%
    \put(0,0){\includegraphics[width=\unitlength,page=2]{ThickenedTorusSfces3.pdf}}%
    \put(0.84184335,0.02239022){\color[rgb]{0,0,0}\makebox(0,0)[lt]{\lineheight{1.25}\smash{\begin{tabular}[t]{l}$T^2\times\{0\}$\end{tabular}}}}%
    \put(0.84694936,0.2138633){\color[rgb]{0,0,0}\makebox(0,0)[lt]{\lineheight{1.25}\smash{\begin{tabular}[t]{l}$T^2\times\{\epsilon_1\}$\end{tabular}}}}%
    \put(0.70653572,0.14238002){\color[rgb]{0,0,0}\makebox(0,0)[lt]{\lineheight{1.25}\smash{\begin{tabular}[t]{l}$N(\alpha)$\end{tabular}}}}%
    \put(0.88077624,0.1232327){\color[rgb]{0,0,0}\makebox(0,0)[lt]{\lineheight{1.25}\smash{\begin{tabular}[t]{l}$N(\partial F)$\end{tabular}}}}%
    \put(0,0){\includegraphics[width=\unitlength,page=3]{ThickenedTorusSfces3.pdf}}%
  \end{picture}%
\endgroup%

\caption{Left: an arc of $\gamma$ cannot run to the same side of $D$. Right: Thus $F\cap (T^2\times[0,\epsilon_1]$ is the union of an annulus and a M\"obius band, for $\epsilon_1$ small.}
\label{Fig:ThickenedTorusSfces3}
\end{figure}

After an isotopy of $F$ that pushes $\alpha$ close to $T^2\times\{0\}$, for small $\epsilon_1>0$ the surface $F \cap T^2\times[0,\epsilon_1]$ consists of $N(\bdy F) \cup N(\alpha)\subset F$. 
This is the union of an annulus $N(\bdy F)$ with a strip $N(\alpha)$ as in \reffig{ThickenedTorusSfces3}, right. The strip must connect one side of the annulus $N(\bdy F)$ to the other, else $\gamma$ would run to the same side of $D$, contradicting the above paragraph or the fact that $\alpha$ intersects only one component $\gamma$ of $\bdy F$. So $N(\bdy F) \cup N(\alpha)$ is nonorientable, and $F\cap T\times[0,\epsilon_1]$ is a M\"obius band with a hole removed. One boundary component lies on $T^2\times\{0\}$, the other on $T^2\times\{\epsilon_1\}$. Observe that adding $N(\alpha)$ to $N(\bdy F)$ adds $-1$ to the Euler characteristic. Observe also that the boundary component on $T^2\times\{\epsilon_1\}$ is obtained from that on $T^2\times\{0\}$ by  surgering a neighborhood of $\beta$. That is, we remove a neighborhood of $\bdy \beta$ from the slope, and attach $\bdy N(\mbox{int}(\beta))$. The result is that the slopes of the curves on $T^2\times\{0\}$ and $T^2\times\{\epsilon_1\}$ have intersection number exactly two on $T^2$.

Repeat the argument, applied to $F\cap T^2\times[\epsilon_1,1]$. If there is another boundary compressing disk toward $T^2\times\{\epsilon_1\}$, then there will be some $\epsilon_2>\epsilon_1$ so that $T^2\times[\epsilon_1,\epsilon_2] \cap F$ is a M\"obius band with a hole. Thus $F$ is obtained by stacking two of these, yielding a nonorientable surface with genus $2$, and unique boundary components on $T^2\times\{0\}$ and $T^2\times\{\epsilon_2\}$, with the slope on $T^2\times\{\epsilon_2\}$ intersecting that on $T^2\times\{\epsilon_1\}$ exactly twice. This process must terminate with some $\epsilon_n$, because the Euler characteristic of $F$ is finite. This implies that $\epsilon_n=1$, and we have constructed a nonorientable surface with two boundary components on $T^2\times\{0\}$ and $T^2\times\{1\}$, as in case~(3).
%This proves that any incompressible surface in $T^2\times I$ is as claimed.

Finally, consider uniqueness of the nonorientable surfaces with fixed boundary components. At each step of the proof, we built a surface that is a M\"obius band with a hole, with one boundary component on $T^2\times\{\epsilon_i\}$ and one on $T^2\times\{\epsilon_{i+1}\}$, and the boundary curves intersecting exactly twice. Adjust the framing on $T^2\times\{0\}$ so that the original slope $F\cap T^2\times\{0\}$ is $0/1$, and together with a choice of slope $1/0$ this forms a basis for the fundamental group of $T^2$.
The construction above then starts with the slope $0/1$, and adds a band to the surface to obtain a new boundary slope of the form $\pm 2/q$, intersecting $0/1$ exactly twice. At each step, we replace a slope with even numerator by one meeting the first exactly twice, hence the result continues to have even numerator.

Recall that slopes on the torus correspond to elements of $\QQ\cup\{1/0\}$. These can be viewed as vertices of the Farey triangulation of $\HH^2$: identifying $\HH^2$ with the upper half plane, the vertices of the Farey triangulation are points of $\QQ\cup\{\infty\}$ on the real line, and the edges run between reduced pairs $p/q$, $r/s$ if and only if $ps-qr=\pm 1$. Here, $ps-qr$ is the (signed) intersection number of the slopes.

We consider slopes with even numerator and draw an edge between them if they have intersection number $\pm 2$. Equivalently, we may consider slopes with odd denominator and the edges between them with intersection number $\pm 1$. This gives a subset of the usual Farey triangulation. Observe that this subset, consisting only of vertices with odd denominators and the edges between them, forms a connected tree; see for example~\cite[Section~3.1]{Bartolini:OneSided}. Thus for any even slope $2p/q$, by following the unique path in this subset of the Farey tree from $0/1$ to $2p/q$, we may build a nonorientable surface in $T^2\times[0,1]$ with slopes $0/1$ on $T^2\times\{0\}$ and $2p/q$ on $T^2\times\{1\}$, where following an edge corresponds to adding an appropriate band to the surface. Observe that traversing an edge, and then returning along it right away, constructs a surface with a compression disk. Thus the construction proceeds monotonically through the Farey tree, and the surface is unique up to isotopy, by uniqueness of the path through the Farey tree. Therefore the pair of slopes uniquely determines the nonorientable surface.
\end{proof}

\bibliography{biblio}
\bibliographystyle{amsplain}

\end{document}